\documentclass[a4paper,10pt]{article}

\usepackage[english]{babel}
\usepackage{amsmath,amsfonts,amssymb,amsthm}
\usepackage{mathrsfs}
\usepackage{bbm}
\usepackage{indentfirst}
\usepackage{mathtools}
\mathtoolsset{showonlyrefs}
\usepackage{graphicx}
\usepackage[font=small]{caption}

\newtheorem{theorem}{Theorem}[section]

\newtheorem{lemma}{Lemma}[section]
\newtheorem{proposition}{Proposition}[section]
\newtheorem{corollary}{Corollary}[section]

\theoremstyle{definition}
\newtheorem{remark}{Remark}[section]

\numberwithin{equation}{section}

\newcommand\blfootnote[1]{\begingroup\renewcommand\thefootnote{}\footnote{#1}\addtocounter{footnote}{-1}\endgroup}

\begin{document}

\title{
{\bf\Large An application of coincidence degree theory to cyclic feedback type systems associated with nonlinear differential operators}
\footnote{Work performed under the auspices of the
Grup\-po Na\-zio\-na\-le per l'Anali\-si Ma\-te\-ma\-ti\-ca, la Pro\-ba\-bi\-li\-t\`{a} e le lo\-ro
Appli\-ca\-zio\-ni (GNAMPA) of the Isti\-tu\-to Na\-zio\-na\-le di Al\-ta Ma\-te\-ma\-ti\-ca (INdAM).
Guglielmo Feltrin is partially supported by the GNAMPA Project 2016
``Problemi differenziali non lineari: esistenza, molteplicit\`{a} e propriet\`{a} qualitative delle soluzioni''.}}
\author{{\bf\large Guglielmo Feltrin}
\vspace{1mm}\\
{\it\small SISSA - International School for Advanced Studies}\\
{\it\small via Bonomea 265}, {\it\small 34136 Trieste, Italy}\\
{\it\small e-mail: guglielmo.feltrin@sissa.it}\vspace{1mm}\\
\vspace{1mm}\\
{\bf\large Fabio Zanolin}
\vspace{1mm}\\
{\it\small Department of Mathematics, Computer Science and Physics, University of Udine}\\
{\it\small via delle Scienze 206},
{\it\small 33100 Udine, Italy}\\
{\it\small e-mail: fabio.zanolin@uniud.it}\vspace{1mm}}

\date{}

\maketitle

\vspace{-2mm}

\begin{center}
\normalsize \textit{Dedicated to Professor Jean Mawhin for his coming 75th birthday}
\end{center}

\vspace{5mm}

\begin{abstract}
\noindent
Using Mawhin's coincidence degree theory, we obtain some new continuation theorems
which are designed to have as a natural application the study of the periodic
problem for cyclic feedback type systems. We also discuss some examples of
vector ordinary differential equations with a $\phi$-Laplacian operator where
our results can be applied.
\end{abstract}

\section{Introduction}\label{section-1}

The aim of this paper is to apply Mawhin's coincidence degree theory in the study of the periodic boundary value problem
for some classes of first order differential systems of cyclic feedback type. From this point of view,
our work continues the research initiated in \cite{CaQiZa-99} and is also partially inspired by the results in \cite{MaMa-98}
on periodic ODE systems with a $\phi$-Laplacian differential operator.

Roughly speaking, by a ``cyclic system'' we usually mean a first order system of ordinary differential equations
where the time evolution of the $j$-th component $y_{j}(t)$ mainly depends upon the pair $(y_{j-1}(t),y_{j}(t))$.
Therefore, the components are ordered in a cyclic manner, so that we consider $n\equiv 0$ for a system of $n$ variables.
Accordingly, such systems usually take a form as
\begin{equation}\label{eq-1.1}
y'_{j} = g_{j}(y_{j-1},y_{j}), \quad j=1,\ldots, n,
\end{equation}
where we agree to interpret $y_{0}$ as $y_{n}$ (cf.~\cite{MPSm-90}).
More general models consider also the case where $g_{j}= g_{j}(y_{j-1},y_{j},y_{j+1})$, which is a typical case
of a system describing nearest neighbor interactions. The term ``feedback'' usually refers to a monotonicity assumption
of the form $\partial g_{j}(y_{j-1},y_{j})/\partial y_{j-1} > 0$ or $\partial g_{j}(y_{j-1},y_{j})/\partial y_{j-1} < 0$,
which reflects the fact that the variable $y_{j-1}$ has a positive or negative effect on the growth of the $j$-th
variable $y_{j}$. These features explain the reason why first order differential systems
with a cyclic feedback structure arise in several
different contexts, both theoretic and applied. As observed in \cite{MPSm-90},
such systems naturally appear in the investigation of biological models (for instance, cellular control systems)
as well as in the study of delay-differential equations or reaction-diffusion equations (after a discretization procedure).

Up to a relabeling of the variables in equation \eqref{eq-1.1}, namely setting $x_{i} := y_{n+1-i}$,
we get an equivalent system of the form
\begin{equation}\label{eq-1.2}
x'_{i} = f_{i}(x_{i},x_{i+1}), \quad i=1,\ldots,n,
\end{equation}
with $f_{i} = g_{j}$ for $i+j=n+1$. In view of the applications that we are going to present in this article,
it will be more convenient for us to consider cyclic systems of the form \eqref{eq-1.2}.
For many concrete examples, in some of the equations of system \eqref{eq-1.2}
there is no dependence of $f_{i}$ upon the $i$-th variable or such dependence can be neglected.
Two typical examples are the following.

Consider a $n$-th order differential equation of the form
\begin{equation}\label{eq-1.3}
x^{(n)} + h(t,x,x',\ldots,x^{(n-1)}) = 0,
\end{equation}
which can be written as
\begin{equation}\label{eq-1.4}
\begin{cases}
\, x'_{i} = x_{i+1}, \quad i= 1,\ldots,n-1, \\
\, x'_{n} = -h(t,x_{1},\ldots,x_{n}).
\end{cases}
\end{equation}
In such a case $f_{i}(x_{i},x_{i+1}) = x_{i+1}$ for $i=1,\ldots,n-1$, so that the first $n-1$ equations
in the cyclic system \eqref{eq-1.2} are strongly simplified. On the other hand,
this example shows that there are cases in which the last equation in \eqref{eq-1.2}
may be more complicated than $x'_{n} = f_{n}(x_{n},x_{1})$.

In some ODE models for population dynamics it is rather common to encounter \textit{Kolmogorov systems} of the form
\begin{equation}\label{eq-1.5}
u'_{i} = u_{i} K_{i}(u_{i+1}), \quad i= 1,\ldots,n\equiv0.
\end{equation}
A typical two-dimensional case is given by the \textit{Lotka-Volterra predator-prey equation}
\begin{equation*}
\begin{cases}
\, u' = u(a - b v) \\
\, v' = v (-c + du).
\end{cases}
\end{equation*}
Since we are looking for positive solutions, we can perform the change of variable $x_{i}(t)= \log(u_{i}(t))$
and transform \eqref{eq-1.5} to the equivalent cyclic feedback system
\begin{equation*}
x'_{i} =  K_{i}(\exp(x_{i+1})), \quad i=1,\ldots,n\equiv 0,
\end{equation*}
which is of the form of \eqref{eq-1.2} with $f_{i}$ independent on the variable $x_{i}$.
This latter model suggests the interest to deal also with the non-autonomous counterpart of system \eqref{eq-1.2},
by assuming an explicit dependence of some of the coefficients on the time variable. This situation naturally
occurs in the study of some Kolmogorov systems, like the Lotka-Volterra one, in which one can consider a seasonal
dependence on the coefficients.

\medskip

In view of the above remarks, we plan to investigate a class of cyclic feedback systems related to \eqref{eq-1.2}
which have a simpler form in the first $n-1$ components but, on the other hand, allow to consider a
more general dependence for the last equation, in order to apply our results to equations of the form \eqref{eq-1.4}
as well. With this respect, we study a system of the form
\begin{equation*}
\begin{cases}
\, x_{1}' = g_{1}(x_{2}) \\
\, x_{2}' = g_{2}(x_{3}) \\
\, \quad \vdots \\
\, x_{n-1}' = g_{n-1}(x_{n}) \\
\, x_{n}' = h(t,x_{1},\ldots,x_{n}),
\end{cases}
\leqno{(\mathscr{C})}
\end{equation*}
where throughout the paper we suppose that $g_{1},\ldots, g_{n-1}$ are continuous functions and $h$
is $T$-periodic in the $t$-variable and satisfies the Carath\'{e}odory assumptions.

A powerful topological tool to produce existence and multiplicity results of periodic solutions
is Mawhin's coincidence degree theory, which allows to apply a topological degree type approach
to problems which can be written as an abstract operator equation of the form $Lx = Nx$,
where $L$ is a linear \textit{non-invertible} operator and $N$ is a nonlinear one acting on a Banach space $X$.
In order to present the next results, we take
\begin{equation*}
X := \mathcal{C}_{T} := \bigl{\{} x\in \mathcal{C}(\mathopen{[}0,T\mathclose{]},\mathbb{R}^{m}) \colon x(0) = x(T)\bigr{\}},
\end{equation*}
with the standard $\sup$-norm $\|\cdot\|_{\infty}$.

In the frame of coincidence degree theory, the main existence result for the periodic problem
\begin{equation*}
\begin{cases}
\, x' = F(t,x)\\
\, x(0) = x(T),
\end{cases}
\leqno{(\mathscr{P})}
\end{equation*}
where $F \colon \mathopen{[}0,T\mathclose{]} \times \mathbb{R}^{m} \to \mathbb{R}^{m}$ is a Carath\'{e}odory vector field,
is \textit{Mawhin's continuation theorem} (cf.~\cite[Th\'{e}or\`{e}me~2]{Ma-69} or \cite[Theorem~4.1]{Ma-93}),
which reads as follows (we denote by ``$\text{\rm deg}_B$'' the Brouwer degree).

\begin{theorem}\label{th-1.1}
Let $\Omega\subseteq X$ be an open bounded set and suppose that:
\begin{itemize}
\item
for each $\lambda\in \mathopen{]}0,1\mathclose{[}$ there is no solution of problem
\begin{equation*}
\begin{cases}
\, x' = \lambda F(t,x) \\
\, x(0) = x(T)
\end{cases}
\end{equation*}
with $x\in \partial\Omega$;
\item
the averaged map $F^{\#} \colon z\mapsto \tfrac{1}{T}\int_{0}^{T} F(t,z)~\!dt$ has no zeros on $\partial\Omega \cap \mathbb{R}^{m}$ and
$\text{\rm deg}_{B}(F^{\#},\Omega\cap\mathbb{R}^{m},0)\neq0$.
\end{itemize}
Then, problem $(\mathscr{P})$ has a solution in $\overline{\Omega}$.
\end{theorem}

A second continuation theorem was proposed in \cite{CaMaZa-92} and extended to delay-differential equations
(with a different proof) in \cite{BaMa-91}. It concerns the case in which the homotopic parameter $\lambda$
is used to modify the original system to an autonomous one. More precisely, we suppose that
there exists a Carath\'{e}odory vector field
$\mathscr{F}=\mathscr{F}(t,x,\lambda)\colon \mathopen{[}0,T\mathclose{]}\times\mathbb{R}^{m}\times \mathopen{[}0,1\mathclose{]} \to \mathbb{R}^{m}$
such that
\begin{equation*}
\mathscr{F}(t,x,0) = F_{0}(x), \qquad \mathscr{F}(t,x,1) = F(t,x).
\end{equation*}
The corresponding existence result can be stated as
follows (cf.~\cite[Theorem~2]{CaMaZa-92}).

\begin{theorem}\label{th-1.2}
Let $\Omega\subseteq X$ be an open bounded set and suppose that:
\begin{itemize}
\item
for each $\lambda\in\mathopen{[}0,1\mathclose{[}$ there is no solution of problem
\begin{equation*}
\begin{cases}
\, x' = \mathscr{F}(t,x,\lambda)\\
\, x(0) = x(T)
\end{cases}
\end{equation*}
with $x\in \partial\Omega$;
\item
$\text{\rm deg}_B(F_{0},\Omega\cap\mathbb{R}^{m},0)\not=0$.
\end{itemize}
Then, problem $(\mathscr{P})$ has a solution in $\overline{\Omega}$.
\end{theorem}

Both these results extend to higher order differential systems of the form \eqref{eq-1.3}. In particular,
both the results and especially Theorem~\ref{th-1.1} have found a great number of applications to
the $T$-periodic problem associated with the vector second order differential equation
\begin{equation*}
u'' + g(t,u,u') = 0,
\end{equation*}
where $g\colon \mathopen{[}0,T\mathclose{]}\times\mathbb{R}^{d}\times \mathbb{R}^{d} \to \mathbb{R}^{d}$ is a Carath\'{e}odory function.

\medskip

The study of ordinary and partial differential equations involving
nonlinear differential operators (like the
$p$-Laplacian, the curvature or the Minkowski operators), started in the mid-twentieth century,
has shown a tremendous growth in the last decades.
Applications of topological degree methods to these equations strongly motivated the
search of new topological tools, such as the continuation theorems for strongly nonlinear operators
(see, for instance, \cite{FuNe-73} for some pioneering works in this direction).
For the periodic boundary value problem associated with non-autonomous ODEs,
Man\'{a}sevich and Mawhin developed in \cite{MaMa-98} new continuation theorems
for the second order vector nonlinear equation
\begin{equation}\label{eq-phi}
\bigl{(}\phi(u')\bigr{)}' + g(t,u,u') = 0.
\end{equation}
New applications were also obtained by the same authors in \cite{MaMa-98T, MaMa-00}
as well as by Mawhin in \cite{Ma-01}.
The two main continuation theorems in \cite{MaMa-98} extend Theorem~\ref{th-1.1} and Theorem~\ref{th-1.2},
respectively, to the above periodic problem, considering, instead of the linear differential operator
$u\mapsto -u''$, the strongly nonlinear operator $u\mapsto - (\phi(u'))'$.
The approach in \cite{MaMa-98} requires that $\phi \colon \mathbb{R}^{d}\to \mathbb{R}^{d}$ is a
homeomorphism with $\phi(0) = 0$ satisfying some additional technical growth conditions
(compare with $(H1)$ and $(H2)$ in Remark~\ref{rem-3.2}).
These continuation theorems concern, respectively, the study of the homotopic equations
\begin{equation}\label{eq-1.6}
\bigl{(}\phi(u')\bigr{)}' + \lambda g(t,u,u') = 0, \quad \lambda\in\mathopen{]}0,1\mathclose{[},
\end{equation}
(in analogy to Theorem~\ref{th-1.1}) or
\begin{equation*}
\bigl{(}\phi(u')\bigr{)}' + \tilde{g}(t,u,u',\lambda) = 0, \quad \lambda\in\mathopen{[}0,1\mathclose{[},
\end{equation*}
with
\begin{equation*}
\tilde{g}(t,u,v,1) = g(t,u,v) \quad \text{ and } \quad\tilde{g}(t,u,v,0) = g_{0}(u,v)
\end{equation*}
(in analogy to Theorem~\ref{th-1.2}).
As far as we know, it seems that the problem whether the technical conditions $(H1)$ and $(H2)$
considered in \cite{MaMa-98,MaMa-98T,MaMa-00,Ma-01}
are necessary or can be removed has not yet been completely solved.
Recently, a different point of view has been considered by Lu and Lu in \cite{LuLu-14}
where, for the mean curvature operator equation, the authors have applied
Mawhin's continuation Theorem~\ref{th-1.1} directly to the first order system
\begin{equation}\label{eq-1.7}
\begin{cases}
\, x_{1}' = \phi^{-1}(x_{2}) \\
\, x_{2}' = -g(t,x_{1},\phi^{-1}(x_{2})).
\end{cases}
\end{equation}
To be more precise, we have to remark that in the equation considered in \cite{LuLu-14} the function $g$
does not depend on $u'$ and therefore the treatment can be simplified.
Clearly, the application of Theorem~\ref{th-1.1} to system \eqref{eq-1.7} involves the study of the parameter-dependent system
\begin{equation*}
\begin{cases}
\, x_{1}' = \lambda \phi^{-1}(x_{2}) \\
\, x_{2}' =  - \lambda g(t,x_{1},\phi^{-1}(x_{2})),
\end{cases}
\quad \lambda\in\mathopen{]}0,1\mathclose{[},
\end{equation*}
which, in turns, corresponds to equation
\begin{equation*}
\bigl{(}\phi(u'/\lambda)\bigr{)}' + \lambda g(t,u,u'/\lambda) = 0, \quad \lambda\in\mathopen{]}0,1\mathclose{[},
\end{equation*}
which looks different from \eqref{eq-1.6} and, apparently, more complicated.

If we write equation \eqref{eq-1.6} as a first order system in $\mathbb{R}^{2d}$, the natural choice would be that of
\begin{equation}\label{eq-1.8}
\begin{cases}
\, x_{1}' = \phi^{-1}(x_{2}) \\
\, x_{2}' =  - \lambda g(t,x_{1},\phi^{-1}(x_{2})), \quad \lambda\in\mathopen{]}0,1\mathclose{[}.
\end{cases}
\end{equation}
The advantage in dealing with such a system is that we \textit{only require} that $\phi$ is an homeomorphism (without
the need of other technical conditions on $\phi$, as considered in \cite{MaMa-98}). On the other hand, Theorem~\ref{th-1.1}
does not apply directly to \eqref{eq-1.8} and this may represent a motivation to try to extend the classical
Mawhin's continuation theorem to a form in which the homotopic parameter $\lambda$ appears only on some components
of the differential system. A first aim of the present paper is to pursue this line of research and, indeed, we
will provide a version of Theorem~\ref{th-1.1} which is suitable for applications to cyclic feedback systems of the
form $(\mathscr{C})$, via an homotopy of the form \eqref{eq-1.8}, when  applied to \eqref{eq-1.7}.
Our proposal for a new continuation theorem is in any case within the setting of Mawhin's coincidence degree theory
and it will be presented as a general theorem for operator equations of coincidence type that
mimics at the abstract level some typical properties of the cyclic systems.

With this respect, the plan of the paper is the following. In Section~\ref{section-2} we  present an application of the theory of
coincidence degree in the setting of a system of operator equations, namely as
coincidence equations involving operators defined in product spaces. We assume the reader familiar with the basics
of Mawhin's coincidence degree, as presented in some classical works like \cite{GaMa-77, Ma-79, Ma-93}.
In any case, when necessary, we shall recall some crucial properties. The key ingredient in our proofs is the
\textit{reduction formula}, a basic tool also in the original applications of the theory (see \cite{Ma-69}
and also \cite{Ma-08,Ma-13} for some recent developments), which allows to relate a Leray-Schauder type degree
in a normed space with a Brouwer degree in a finite-dimensional space. Accordingly, our main results in
Section~\ref{section-2} are Theorem~\ref{th-2.1}, where we perform an abstract homotopy of the form \eqref{eq-1.8},
and the subsequent Lemma~\ref{lem-2.4}, where we provide a precise formula for
the computation of the degree. Such results have an immediate application to the periodic problem
and therefore in Section~\ref{section-3} we give some existence theorems of continuation type for system $(\mathscr{C})$
which are analogous to Theorem~\ref{th-1.1} (see Theorem~\ref{th-cycl1} and Theorem~\ref{th-cycl2}). In the same section
we produce, as a consequence of our main results, a continuation theorem for $\phi$-Laplacian differential systems of the form
\eqref{eq-phi} which involves the study of \eqref{eq-1.6}. Our contribution for this kind of equations is
Theorem~\ref{th-mama1} which is
exactly Man\'{a}sevich-Mawhin theorem \cite[Theorem~3.1]{MaMa-98},
with the only difference that with our approach we can avoid some technical
conditions on $\phi$ which were assumed in \cite{MaMa-98}. We also provide an example of a general type of
higher dimensional $\phi$-map where our result can be applied. In Section~\ref{section-4} we discuss a second type
of continuation theorems which are essentially based on Theorem~\ref{th-1.2}. More precisely, we consider the case
when the admissible homotopy transforms a non-autonomous system of the form $(\mathscr{C})$ into an autonomous one
where the coincidence degree can be computed using the theorems in \cite{BaMa-91} and \cite{CaMaZa-92}
(these auxiliary results are recalled in a final appendix together with a more general version suitable for our applications).
Again our purpose is to show that, when a given system allows an
equivalent representation in the cyclic feedback form, some continuation theorems can be reformulated in a very
effective fashion, thus avoiding some additional technical conditions. Applications are given again to
$\phi$-Laplacian differential systems of the form \eqref{eq-phi} and
our contribution Theorem~\ref{th-mama2} is precisely Man\'{a}sevich-Mawhin theorem \cite[Theorem~4.1]{MaMa-98}
(without extra assumptions on the $\phi$-operator). Next, in
Section~\ref{section-5} we reconsider for a broader class of differential operators Hartman-Knobloch theorem, recently extended by Mawhin in \cite{Ma-00} and by Mawhin and Ure\~{n}a in
\cite{MaUr-02} to $p$-Laplacian systems.
Finally, in Section~\ref{section-6} we briefly discuss some possible extensions to operators which are not defined on the whole space.

\medskip

We conclude this introductory section presenting a few notation used in the present paper.
In the $N$-dimensional real Euclidean space $\mathbb{R}^{N}$ we denote by $\langle \cdot,\cdot \rangle$ the standard inner product and
by $\|\cdot\|_{\mathbb{R}^N}$ the corresponding norm. When no confusion may occur we shall also use the symbol $|\cdot|$ as
a simplified notation for the norm. If we consider a homeomorphism $\phi \colon \mathcal{A} \to \mathcal{B}$, we always implicitly assume that
$\phi(\mathcal{A}) = \mathcal{B}$.
Thus, in particular, for a homeomorphism $\phi \colon \mathbb{R}^{N} \to \mathbb{R}^{N}$, we suppose that
$\phi(\mathbb{R}^{N}) = \mathbb{R}^{N}$.
In any fixed (finite or infinite dimensional) normed space, we denote by $B(x_{0},r)$ (respectively, $B[x_0,r]$)
the open (respectively, closed) ball of center a point $x_{0}$ and radius $r > 0$.
We denote by ``$\text{\rm deg}_{B}$'' the finite-dimensional Brouwer degree and by ``$\text{\rm deg}_{LS}$'' the Leray-Schauder degree
in the context of locally compact operators on arbitrary open not necessarily bounded sets (cf.~\cite{Nu-85,Nu-93} for a precise definition).
Finally, we denote by ``$D_{L}$'' the coincidence degree extended to locally compact operators (cf.~\cite[Appendix~A]{FeZa-pp2015}).

\section{Coincidence degree theory in product spaces}\label{section-2}

Throughout the section, when not otherwise specified, $i$ is a generic index from $1$ to $n$.

For $i=1,\ldots,n$, let $X_{i},Z_{i}$ be real normed linear spaces and let
\begin{equation*}
L_{i} \colon \text{\rm dom}\,L_{i} (\subseteq X_{i}) \to Z_{i}
\end{equation*}
be a linear Fredholm mapping of index zero, i.e.~$\text{\rm Im}\,L_{i}$ is a closed subspace of $Z_{i}$ and
$\text{\rm dim}(\ker L_{i}) = \text{\rm codim}(\text{\rm Im}\,L_{i})$ are finite.
We denote by $\ker L_{i} = L_{i}^{-1}(0)\subseteq X_{i}$ the kernel of $L_{i}$, by $\text{\rm Im}\,L_{i}\subseteq Z_{i}$ the image of $L_{i}$
and by $\text{\rm coker}\,L_{i} = Z_{i}/\text{\rm Im}\,L_{i}$ the quotient space of $Z_{i}$ under the equivalence
relation $w_{1} \sim w_{2}$ if and only if $w_{1}-w_{2} \in \text{\rm Im}\,L_{i}$.
Thus $\text{\rm coker}\,L_{i}$ is a complementary subspace of $\text{\rm Im}\,L_{i}$ in $Z_{i}$.

From basic results of linear functional analysis, due to the fact that $L_{i}$ is a Fredholm mapping,
there exist linear continuous projections
\begin{equation*}
P_{i} \colon X_{i} \to \ker L_{i}, \qquad Q_{i} \colon Z_{i} \to \text{\rm coker}\,L_{i}
\end{equation*}
so that
\begin{equation*}
X_{i} = \ker L_{i} \oplus \ker P_{i}, \qquad Z_{i} = \text{\rm Im}\,L_{i} \oplus \text{\rm Im}\,Q_{i}.
\end{equation*}
We denote by
\begin{equation*}
K_{i} \colon \text{\rm Im}\,L_{i} \to \text{\rm dom}\,L_{i} \cap \ker P_{i}
\end{equation*}
the right inverse of $L_{i}$, i.e.~$L_{i} K_{i}(v) = v$ for each $v\in \text{\rm Im}\,L_{i}$.
Since $\ker L_{i}$ and $\text{\rm coker}\,L_{i}$ are finite-dimensional vector spaces of the same dimension,
once an orientation on both spaces is fixed,
we choose a linear orientation-preserving isomorphism $J_{i} \colon \text{\rm coker}\,L_{i} \to \ker L_{i}$.

\medskip

Let us consider the product spaces
\begin{equation*}
X:=\prod_{i=1}^{n} X_{i}, \qquad Z:= \prod_{i=1}^{n} Z_{i},
\end{equation*}
with the usual norms.

Setting $\text{\rm dom}\,L:=\prod_{i=1}^{n} \text{\rm dom}\,L_{i}$, we define $L \colon \text{\rm dom}\,L (\subseteq X) \to Z$ as
\begin{equation*}
L(u):=(L_{1}(u_{1}),\ldots,L_{n}(u_{n})), \quad u=(u_{1},\ldots,u_{n})\in \text{\rm dom}\,L, \quad \text{with } u_{i}\in X_{i}.
\end{equation*}
It is obvious to verify that $L$ is a linear Fredholm mapping of index zero.
Next, we observe that
\begin{equation*}
\ker L = L^{-1}(0) = \prod_{i=1}^{n} \ker L_{i} \subseteq X \quad \text{ and } \quad \text{\rm Im}\,L = \prod_{i=1}^{n} \text{\rm Im}\,L_{i} \subseteq Z
\end{equation*}
are the kernel of $L$ and the image of $L$, respectively.
Finally, we define the map $K \colon \text{\rm Im}\,L \to \text{\rm dom}\,L \cap \prod_{i=1}^{n}\ker P_{i}$ as
\begin{equation*}
K(v):=(K_{1}(v_{1}),\ldots,K_{n}(v_{n})), \quad v=(v_{1},\ldots,v_{n})\in \text{\rm Im}\,L, \quad \text{with } v_{i}\in Z_{i}.
\end{equation*}
It is easy to check that $K$ is the right inverse of $L$.

Let also define $P \colon X \to \ker L$, $P(u):=(P_{1}(u_{1}),\ldots,P_{n}(u_{n}))$, for $u=(u_{1},\ldots,u_{n})\in X$ (with $u_{i}\in X_{i}$),
$Q \colon Z \to \text{\rm coker}\,L$, $Q(v):=(Q_{1}(v_{1}),\ldots,Q_{n}(v_{n}))$, for $v=(v_{1},\ldots,v_{n})\in Z$ (with $v_{i}\in Z_{i}$),
and $J \colon \text{\rm coker}\,L \to \ker L$, $J(v):=(J_{1}(v_{1}),\ldots,J_{n}(v_{n}))$, for $v=(v_{1},\ldots,v_{n})\in  \text{\rm coker}\,L$ (with $v_{i}\in Z_{i}$).

\medskip

Let
\begin{equation*}
N \colon \text{\rm dom}\,N (\subseteq X) \to Z
\end{equation*}
be a nonlinear \textit{$L$-completely continuous} operator, namely $N$ and $K(Id_{Z}-Q)N$ are continuous,
and also
$QN(B)$ and $K(Id_{Z}-Q)N(B)$ are relatively compact sets, for each bounded set $B\subseteq \text{\rm dom}\,N$.
For example, $N$ is $L$-completely continuous when $N$ is continuous, maps bounded sets to bounded sets
and $K$ is a compact linear operator.

We further define $N_{i} \colon \text{\rm dom}\,N \to Z_{i}$ as
\begin{equation*}
N_{i}(u) := (\pi^{Z}_{i}\circ N) (u), \quad u\in \text{\rm dom}\,N,
\end{equation*}
where $\pi^{Z}_{i} \colon Z \to Z_{i}$ is the standard projection.

In the sequel, to simplify the notation, we will write $Lu$ and
$Nu$ in place of $L(u)$ and $N(u)$, respectively. The same
convention will be used also for other operators.

\medskip

Now we consider the \textit{coincidence equation}
\begin{equation*}
Lu = Nu,\quad u\in \text{\rm dom}\,L\cap \text{\rm dom}\,N,
\end{equation*}
which can be equivalently written as a system
\begin{equation}\label{coinc-eq}
\begin{cases}
\, L_{i}u_{i} = N_{i}(u_{1},\ldots,u_{n}),\quad u=(u_{1},\ldots,u_{n})\in \text{\rm dom}\,L\cap\text{\rm dom}\,N,\\
\, i=1,\ldots,n.
\end{cases}
\end{equation}

From Mawhin's coincidence degree theory, one can see that system \eqref{coinc-eq} is equivalent to the fixed point problem
\begin{equation*}
u = \Phi(u), \quad u\in \text{\rm dom}\,N,
\end{equation*}
where $\Phi = \Phi_{N} \colon \text{\rm dom}\,N \to X$ is defined as
\begin{equation}\label{eq-Phi}
\Phi(u):= Pu + JQNu + K(Id_{Z}-Q)Nu, \quad u\in \text{\rm dom}\,N.
\end{equation}
Hence $\Phi$ is of the form $\Phi=(\Phi_{1},\ldots,\Phi_{n})$, with $\Phi_{i} \colon \text{\rm dom}\,N \to X_{i}$ given by
\begin{equation*}
\Phi_{i}(u):= P_{i}u_{i} + J_{i}Q_{i}N_{i}u + K_{i}(Id_{Z_{i}}-Q_{i})N_{i}u, \quad u=(u_{1},\ldots,u_{n})\in \text{\rm dom}\,N.
\end{equation*}
Notice that, under the above assumptions, $\Phi \colon \text{\rm dom}\,N \to X$ is a completely continuous operator.

\medskip

As a first step, we state the classical homotopic invariance property of Mawhin's coincidence degree.
We present the theory in a slightly simplified version than the more general one
developed in \cite{Nu-85,Nu-93} for locally compact operators (see Remark~\ref{rem-2.1}
for a more general statement).

We recall that the coincidence degree $D_{L}(L-N,\Omega)$ of $L$ and $N$ in $\Omega$ is defined as $\text{\rm deg}_{LS}(Id_{X}-\Phi,\Omega,0)$,
for $\Phi = \Phi_{N}$ as in \eqref{eq-Phi}.

\begin{lemma}[Homotopic invariance]\label{lemma-inv}
Let $L$ and $N$ be as above and let $\Omega\subseteq \text{\rm dom}\,N$ be an open (possibly unbounded) set.
Suppose that there exists an $L$-completely continuous map $\tilde{N} \colon \Omega\times\mathopen{[}0,1\mathclose{]} \to Z$ such that
\begin{equation*}
\tilde{N}(u,0)=\hat{N}(u), \quad \tilde{N}(u,1)=N(u), \quad \forall \, u\in \Omega,
\end{equation*}
where $\hat{N} \colon \Omega \to Z$.
Moreover, suppose that the set
\begin{equation*}
\mathcal{S}:=\bigcup_{\vartheta\in\mathopen{[}0,1\mathclose{]}}\bigl{\{}u\in \Omega \cap \text{\rm dom}\,L \colon Lu=\tilde{N}(u,\vartheta) \bigr{\}}
\end{equation*}
is a compact subset of $\Omega$.
Then, the map
\begin{equation*}
\vartheta \mapsto D_{L}(L-\tilde{N}(\cdot,\vartheta),\Omega)
\end{equation*}
is well-defined and constant on $\mathopen{[}0,1\mathclose{]}$. In particular, it holds that
\begin{equation*}
D_{L}(L-N,\Omega) = D_{L}(L-\hat{N},\Omega).
\end{equation*}
\end{lemma}

\begin{remark}\label{rem-2.1}
In Lemma~\ref{lemma-inv} we have stated the homotopy invariance for an $L$-completely continuous map $\tilde{N}$.
We stress that the same conclusion holds for a continuous map $\tilde{N}$ such that the set $\mathcal{S}$ is compact
and there exists a bounded open neighborhood $\mathcal{W}$ of $\mathcal{S}$ such that $\overline{\mathcal{W}}\subseteq \Omega$ and
$(K(Id_{Z}-Q)\tilde{N})|_{\mathopen{[}0,1\mathclose{]}\times\overline{\mathcal{W}}}$ is a compact map.
$\hfill\lhd$
\end{remark}

\medskip

Let $\pi^{X}_{i} \colon X \to X_{i}$ be the standard projection.
If $\Omega\subseteq X$, we define
\begin{equation*}
\Omega_{i} := \pi^{X}_{i}(\Omega).
\end{equation*}
We observe that, if $\Omega$ is open in $X$, then $\Omega_{i}$ is open in $X_{i}$, for all $i=1,\ldots,n$.

\medskip

Now we recall the \textit{Reduction Formula} of the Leray-Schauder degree for locally compact operators,
which is a direct consequence of the \textit{Commutativity property} (cf.~\cite[pp.~26--27]{Nu-85} and \cite[pp.~148--149]{Nu-93}).
This property will be crucial in the proof of a subsequent result (cf.~Lemma~\ref{lem-2.3}).

\begin{lemma}[Reduction Formula]\label{red-for}
Let $X$ be a normed linear space. Let $U \subseteq X$ be an open (possibly unbounded) set. Let $\Psi\colon U\to X$ be a continuous map such that $\text{\rm deg}_{LS}(Id_{X}-\Psi, U,0)$ is defined.
Let $Y\subseteq X$ be a subspace such that $\Psi(U)\subseteq Y$.
Then
\begin{equation*}
\text{\rm deg}_{LS}(Id_{X}-\Psi, U,0) = \text{\rm deg}_{LS}(Id_{Y}-\Psi|_{Y}, U\cap Y,0).
\end{equation*}
\end{lemma}

In the statement, we implicitly identify $\Psi$ with $j\circ\Psi$,
where $j\colon Y \to X$ is the (continuous) inclusion.

The following result is an application of Lemma~\ref{red-for}.

\begin{lemma}\label{lem-2.3}
Let $\Omega\subseteq X$ be an open (possibly unbounded) set.
Let $L$ be as above and $\hat{N} \colon \Omega \to Z$ be an $L$-completely continuous operator.
Suppose that $\hat{N}u$, $u=(u_{1},\ldots,u_{n})\in \Omega$, has components of the following form
\begin{equation*}
\begin{cases}
\, \hat{N}_{i}(u_{i+1}), & \text{for } \, i=1,\ldots,n-1; \\
\, \hat{N}_{n}(u_{1},\ldots,u_{n}).
\end{cases}
\end{equation*}
Assume that
\begin{equation}\label{cond-i}
\text{\rm Im}\,L_{i}\cap \hat{N}_{i}(\Omega_{i+1} \cap \ker L_{i+1}) \subseteq \{0_{Z_{i}}\}, \quad \text{for all } \, i=1,\ldots,n-1,
\end{equation}
and
\begin{equation}\label{cond-n}
\text{\rm Im}\,L_{n}\cap \hat{N}_{n}(\Omega\cap\text{\rm dom}\,L) \subseteq \{0_{Z_{n}}\}.
\end{equation}
Moreover, assume that
\begin{equation*}
\bigl{\{} u\in \Omega\cap\ker L \colon \hat{N}u = 0_{Z} \bigr{\}}
\end{equation*}
is a compact subset of $\Omega$.
Then
\begin{equation*}
D_{L}(L-\hat{N},\Omega) = \text{\rm deg}_{B}(\mathcal{N},\Omega \cap \ker L,0),
\end{equation*}
where $\mathcal{N}\colon \Omega\cap \ker L \to \ker L$ is defined as
\begin{equation*}
\mathcal{N}:=\bigl{(}-J_{1}Q_{1}\hat{N}_{1}|_{\Omega_{2}\cap\ker L_{2}},
\ldots,-J_{n-1}Q_{n-1}\hat{N}_{n-1}|_{\Omega_{n}\cap\ker L_{n}},
-J_{n}Q_{n}\hat{N}_{n}|_{\Omega\cap\ker L} \bigr{)}.
\end{equation*}
\end{lemma}

\begin{proof}
First of all, we introduce the operator $\tilde{\Phi} \colon \mathopen{[}0,1\mathclose{]} \times \Omega \to X$
of the form $\tilde{\Phi}(\vartheta,u)=\tilde{\Phi}^{\vartheta}(u)=(\tilde{\Phi}^{\vartheta}_{1}(u),\ldots,\tilde{\Phi}^{\vartheta}_{n}(u))$,
with $\tilde{\Phi}^{\vartheta}_{i} \colon \Omega \to X_{i}$ defined as
\begin{equation*}
\begin{cases}
\, \tilde{\Phi}^{\vartheta}_{i}(u) := P_{i}u_{i} + J_{i}Q_{i}\hat{N}_{i}u_{i+1} + \vartheta K_{i}(Id_{Z_{i}}-Q_{i})\hat{N}_{i}u_{i+1}, & i=1,\ldots,n-1; \\
\, \tilde{\Phi}^{\vartheta}_{n}(u) := P_{n}u_{n} + J_{n}Q_{n}\hat{N}_{n}u + \vartheta K_{n}(Id_{Z_{n}}-Q_{n})\hat{N}_{n}u.
\end{cases}
\end{equation*}
We stress that $\tilde{\Phi}$ is the completely continuous operator associated with the coincidence equation
\begin{equation*}
Lu = \vartheta\hat{N}u,\quad u\in \Omega \cap \text{\rm dom}\,L, \quad \vartheta \in \mathopen{[}0,1\mathclose{]},
\end{equation*}
in the sense that $u\in \Omega$ is such that $u=\tilde{\Phi}^{\vartheta}(u)$ for some $\vartheta\in \mathopen{]}0,1\mathclose{]}$
if and only if $u\in \Omega \cap \text{\rm dom}\,L$ and $Lu = \vartheta\hat{N}u$.

We claim that the set
\begin{equation*}
\tilde{\mathcal{S}}:=\bigcup_{\vartheta\in\mathopen{[}0,1\mathclose{]}}\bigl{\{}u\in \Omega \colon u=\tilde{\Phi}^{\vartheta}(u) \bigr{\}}
\end{equation*}
is a compact subset of $\Omega$.

Let us fix an arbitrary $\vartheta\in\mathopen{]}0,1\mathclose{]}$.
If $u\in\Omega$ is such that $u=\tilde{\Phi}^{\vartheta}(u)$, then, in particular, from the last equation it holds that
\begin{equation*}
u_{n} = P_{n}u_{n} + J_{n}Q_{n}\hat{N}_{n}u + \vartheta K_{n}(Id_{Z_{n}}-Q_{n})\hat{N}_{n}u,
\end{equation*}
so that
\begin{equation*}
u\in \Omega\cap\text{\rm dom}\,L \quad
(u_{n}\in \Omega_{n}\cap\text{\rm dom}\,L_{n}) \quad \text{ and } \quad L_{n}\dfrac{u_{n}}{\vartheta}=\hat{N}_{n}u.
\end{equation*}
From hypothesis \eqref{cond-n}, we easily obtain that
\begin{equation*}
u_{n} \in \Omega_{n} \cap \ker L_{n} \quad \text{ and } \quad \hat{N}_{n}u=0_{Z_{n}}.
\end{equation*}
Next, considering the $(n-1)$-component of the equality $u=\tilde{\Phi}^{\vartheta}(u)$, we can write
\begin{equation*}
u_{n-1} = P_{n-1}u_{n-1} + J_{n-1}Q_{n-1}\hat{N}_{n-1}u_{n} + \vartheta K_{n-1}(Id_{Z_{n-1}}-Q_{n-1})\hat{N}_{n-1}u_{n},
\end{equation*}
so that
\begin{equation*}
u_{n-1}\in \Omega_{n-1}\cap\text{\rm dom}\,L_{n-1} \quad \text{ and } \quad  L_{n-1}\dfrac{u_{n-1}}{\vartheta}=\hat{N}_{n-1}u_{n}.
\end{equation*}
Taking into account that $u_{n} \in \Omega_{n} \cap \ker L_{n}$, hypothesis \eqref{cond-i} (with $i=n-1$) ensures that
\begin{equation*}
u_{n-1} \in \Omega_{n-1} \cap \ker L_{n-1} \quad \text{ and } \quad \hat{N}_{n-1}u_{n}=0_{Z_{n-1}}.
\end{equation*}
Proceeding in this way (by repeating inductively the same argument), hypothesis \eqref{cond-i} ensures that
\begin{equation*}
u_{i} \in \Omega_{i} \cap \ker L_{i} \quad \text{ and } \quad \hat{N}_{i}u_{i}=0_{Z_{i}}, \quad \text{for all } \, i=1,\ldots,n-1.
\end{equation*}

Next, let $\vartheta=0$.
If $u\in\Omega$ is such that $u=\tilde{\Phi}^{0}(u)$, then
\begin{equation*}
\begin{cases}
\, u_{i} = P_{i}u_{i} + J_{i}Q_{i}\hat{N}_{i} u_{i+1}, & i=1,\ldots,n-1; \\
\, u_{n} = P_{n}u_{n} + J_{n}Q_{n}\hat{N}_{n}u.
\end{cases}
\end{equation*}
We immediately deduce that
\begin{equation*}
u_{i} \in \Omega_{i} \cap \ker L_{i} \quad \text{for all } \, i=1,\ldots,n,
\end{equation*}
and so $u_{i} = P_{i}u_{i}$.
Therefore, $Q_{i}\hat{N}_{i}u_{i+1}=0_{Z_{i}}$, for $i=1,\ldots,n-1$, and $Q_{n}\hat{N}_{n}u=0_{Z_{n}}$.
Then, $\hat{N}_{i}u_{i+1}\in \text{\rm Im}\,L_{i}$, for $i=1,\ldots,n-1$, and $\hat{N}_{n}u\in \text{\rm Im}\,L_{n}$. Hence, we obtain that
\begin{equation*}
\hat{N}_{i}u_{i+1} \in \text{\rm Im}\,L_{i}\cap \hat{N}_{i}(\Omega_{i+1} \cap \ker L_{i+1}) = \{0_{Z_{i}}\}, \quad \text{for all } \, i=1,\ldots,n-1;
\end{equation*}
and
\begin{equation*}
\hat{N}_{n}u \in \text{\rm Im}\,L_{n}\cap \hat{N}_{n}(\Omega\cap\text{\rm dom}\,L) = \{0_{Z_{n}}\}.
\end{equation*}
Then
\begin{equation*}
\hat{N}u = 0_{Z}.
\end{equation*}

By the above observations, we have that $\tilde{\mathcal{S}} \subseteq \{ u\in \Omega\cap\ker L \colon \hat{N}u = 0_{Z} \}$.
Since the converse inclusion is trivial, we obtain that
\begin{equation*}
\tilde{\mathcal{S}} = \bigl{\{} u\in \Omega\cap\ker L \colon \hat{N}u = 0_{Z} \bigr{\}}.
\end{equation*}
Therefore, by the hypothesis, $\tilde{\mathcal{S}}$ is a compact subset of $\Omega$.

The homotopy invariance of the Leray-Schauder degree for locally compact operators implies that the map
\begin{equation*}
\vartheta \mapsto \text{\rm deg}_{LS}(Id_{X}-\tilde{\Phi}^{\vartheta},\Omega,0)
\end{equation*}
is well-defined and constant on $\mathopen{[}0,1\mathclose{]}$. In particular, it holds that
\begin{equation*}
\text{\rm D}_{L}(L-\hat{N},\Omega)
= \text{\rm deg}_{LS}(Id_{X}-\tilde{\Phi}^{1},\Omega,0)
= \text{\rm deg}_{LS}(Id_{X}-\tilde{\Phi}^{0},\Omega,0),
\end{equation*}
where
\begin{equation*}
\tilde{\Phi}^{0} = Pu+JQ\hat{N} \colon \Omega \to \ker L.
\end{equation*}

Finally, we notice that $\tilde{\Phi}^{0}(\Omega)\subseteq \ker L$ and recall that $\ker L$ is a finite-dimensional subspace of $X$.
Therefore, we can apply the \textit{Reduction Formula} (i.e.~Lemma~\ref{red-for} with $U=\Omega$, $\Psi=\tilde{\Phi}^{0}$ and $Y=\ker L$) and we conclude that
\begin{equation*}
\begin{aligned}
\text{\rm deg}_{LS}(Id_{X}-\tilde{\Phi}^{0},\Omega,0)
&= \text{\rm deg}_{LS}(Id_{\ker L}-\tilde{\Phi}^{0}|_{\ker L},\Omega\cap \ker L,0)
\\ &= \text{\rm deg}_{B}(\mathcal{N},\Omega \cap \ker L,0).
\end{aligned}
\end{equation*}
The thesis immediately follows.
\end{proof}

\begin{remark}\label{rem-2.2}
We underline that the same conclusion of Lemma~\ref{lem-2.3} holds if in the statement we replace conditions
\eqref{cond-i} and \eqref{cond-n} with
\begin{equation*}
\hat{N}_{i}(\Omega_{i+1} \cap \ker L_{i+1}) \subseteq \text{\rm coker}\,L_{i}, \quad \text{for all } \, i=1,\ldots,n-1,
\end{equation*}
and
\begin{equation*}
\hat{N}_{n}(\Omega) \subseteq \text{\rm coker}\,L_{n},
\end{equation*}
respectively.
$\hfill\lhd$
\end{remark}

\medskip

Combining Lemma~\ref{lemma-inv} and Lemma~\ref{lem-2.3}, we obtain the following result.

\begin{theorem}\label{th-2.1}
Let $L$ be as above and let $M\colon \text{\rm dom}\,M (\subseteq X) \to Z$ be a nonlinear $L$-completely continuous operator.
Suppose that $M u$, $u=(u_{1},\ldots,u_{n})\in \text{\rm dom}\,M$, has components of the following form
\begin{equation*}
\begin{cases}
\, M_{i}(u_{i+1}), & \text{for } \, i=1,\ldots,n-1; \\
\, M_{n}(u_{1},\ldots,u_{n}).
\end{cases}
\end{equation*}
For $\vartheta\in\mathopen{]}0,1\mathclose{]}$, consider the following coincidence system
\begin{equation*}
\begin{cases}
\, L_{i}u_{i} = M_{i}(u_{i+1}), & i=1,\ldots,n-1; \\
\, L_{n}u_{n} = \vartheta  M_{n}(u_{1},\ldots,u_{n}).
\end{cases}
\leqno{(\mathscr{P}_{\vartheta})}
\end{equation*}
Let $\Omega\subseteq \text{\rm dom}\,M$ be an open (possibly unbounded) set. We define
\begin{equation*}
\mathcal{S}_{\vartheta}:=\{u\in\Omega\cap \text{\rm dom}\,L \colon u \text{ is a solution of } (\mathscr{P}_{\vartheta})\},
\quad \vartheta\in\mathopen{]}0,1\mathclose{]}.
\end{equation*}
Assume that
\begin{itemize}
\item [$(i)$] $\text{\rm Im}\,L_{i}\cap M_{i}(\Omega_{i+1} \cap \ker L_{i+1}) \subseteq \{0_{Z_{i}}\}$, for all $i=1,\ldots,n-1$;
\item [$(ii)$] there exists a compact set $\mathcal{K}\subseteq\Omega$ such that $\mathcal{S}_{\vartheta}\subseteq\mathcal{K}$, for all $\vartheta\in\mathopen{]}0,1\mathclose{]}$;
\item [$(iii)$] the set $\mathcal{S}_{0}:=\{u\in\Omega\cap \ker L \colon QMu=0\}$ is compact.
\end{itemize}
Then
\begin{equation*}
D_{L}(L-M,\Omega) = \text{\rm deg}_{B}(\mathcal{M},\Omega \cap \ker L,0),
\end{equation*}
where $\mathcal{M}\colon \Omega\cap \ker L \to \ker L$ is defined as
\begin{equation*}
\mathcal{M}:=\bigl{(}-J_{1}Q_{1}M_{1}|_{\Omega_{2}\cap\ker L_{2}},
\ldots,-J_{n-1}Q_{n-1}M_{n-1}|_{\Omega_{n}\cap\ker L_{n}},
-J_{n}Q_{n}M_{n}|_{\Omega\cap\ker L} \bigr{)}.
\end{equation*}
\end{theorem}

\begin{proof}
For $\vartheta\in\mathopen{[}0,1\mathclose{]}$, we define the auxiliary coincidence system
\begin{equation*}
\begin{cases}
\, L_{i}u_{i} = M_{i}(u_{i+1}), & i=1,\ldots,n-1; \\
\, L_{n}u_{n} = \vartheta M_{n}u +(1-\vartheta) Q_{n}M_{n} u
\end{cases}
\leqno{(\mathscr{A}_{\vartheta})}
\end{equation*}
and the set
\begin{equation*}
\mathcal{S}'_{\vartheta}:=\{u\in\Omega\cap \text{\rm dom}\,L \colon u \text{ is a solution of } (\mathscr{A}_{\vartheta})\}.
\end{equation*}
First of all, we observe that
\begin{equation*}
\mathcal{S}'_{\vartheta} = \mathcal{S}_{\vartheta}, \quad \text{for all } \, \vartheta\in\mathopen{]}0,1\mathclose{]},
\end{equation*}
namely $u$ is a solution of $(\mathscr{A}_{\vartheta})$
if and only if $u$ is a solution of $(\mathscr{P}_{\vartheta})$.
Indeed, if $u$ is a solution of $(\mathscr{A}_{\vartheta})$, applying the projection $Q_{n}$ to the last equation of $(\mathscr{A}_{\vartheta})$, we obtain $Q_{n}M_{n}u= 0$;
then clearly $u$ solves $(\mathscr{P}_{\vartheta})$.
On the other hand, if $u$ is a solution of $(\mathscr{P}_{\vartheta})$, applying the projection $Q_{n}$ to the last equation of $(\mathscr{P}_{\vartheta})$,
we obtain $\vartheta Q_{n}M_{n}u= 0$ (with $\vartheta \neq 0$); then we deduce that $u$ solves $(\mathscr{A}_{\vartheta})$.
Secondly, we notice that
\begin{equation}\label{eq-S0}
\mathcal{S}'_{0} \subseteq \mathcal{S}_{0}.
\end{equation}
Indeed, if $u$ is a solution of $(\mathscr{A}_{0})$, then $L_{n} u_{n} =
Q_{n} M_{n} u$. Hence, $u_{n} \in \ker L_{n}$ and $Q_{n} M_{n} u = 0$.
Therefore, using condition $(i)$, we find that
$u\in\ker L$ and also $QM u = 0$. In this manner \eqref{eq-S0} is proved.

Let $\tilde{M} \colon \Omega\times\mathopen{[}0,1\mathclose{]} \to Z$ be the continuous homotopy with components
\begin{equation*}
\begin{cases}
\, M_{i}(u_{i+1}), & \text{for } \,  i=1,\ldots,n-1; \\
\, \vartheta M_{n}u +(1-\vartheta) Q_{n}M_{n} u,
\end{cases}
\end{equation*}
where $u=(u_{1},\ldots,u_{n})\in \Omega$ and $\vartheta\in\mathopen{[}0,1\mathclose{]}$.
We observe that $\tilde{M}(u,1)=Mu$ and $\tilde{M}(u,0)$ has $Q_{n}M_{n} u$ as last component (which is finite-dimensional).

\smallskip

We divide the proof into two steps.

\smallskip

\noindent
\textit{Step~1. } We claim that
\begin{equation*}
D_{L}(L-M,\Omega) = D_{L}(L-\tilde{M}(\cdot,0),\Omega).
\end{equation*}
We first notice that, from condition $(ii)$ and the above remark, we have $\mathcal{S}'_{\vartheta}\subseteq\mathcal{K}$, for all $\vartheta\in\mathopen{]}0,1\mathclose{]}$.
Recalling also \eqref{eq-S0} together with condition $(iii)$, we find that
\begin{equation*}
\bigcup_{\vartheta\in\mathopen{[}0,1\mathclose{]}}\bigl{\{}u\in \Omega \cap \text{\rm dom}\,L \colon Lu=\tilde{M}(u,\vartheta) \bigr{\}}
= \bigcup_{\vartheta\in\mathopen{[}0,1\mathclose{]}} \mathcal{S}'_{\vartheta}
\end{equation*}
is a compact subset of $\Omega$, since closed and contained in the compact set $ \mathcal{K}\cup\mathcal{S}_{0}$.
Therefore we can apply Lemma~\ref{lemma-inv}. The claim is thus proved.

\smallskip

\noindent
\textit{Step~2. } We claim that
\begin{equation*}
D_{L}(L-\tilde{M}(\cdot,0),\Omega) = \text{\rm deg}_{B}(\mathcal{M},\Omega \cap \ker L,0).
\end{equation*}
We are going to apply Lemma~\ref{lem-2.3} to the $L$-completely continuous operator $\hat{N} \colon \Omega \to Z$ defined in this way:
the components of $\hat{N}u$, $u=(u_{1},\ldots,u_{n})\in \Omega$, have the following form
\begin{equation*}
\begin{cases}
\, \hat{N}_{i}(u_{i+1}) := M_{i}(u_{i+1}), & \text{for } \,  i=1,\ldots,n-1; \\
\, \hat{N}_{n}(u_{1},\ldots,u_{n}) := Q_{n}M_{n}(u_{1},\ldots,u_{n}).
\end{cases}
\end{equation*}
Clearly condition \eqref{cond-i} of Lemma~\ref{lem-2.3} corresponds to condition $(i)$. Moreover, hypothesis \eqref{cond-n} is satisfied, since by the definition of $\hat{N}_{n}$
it holds that $\hat{N}_{n}(\Omega)\subseteq \text{\rm coker}\,L_{n}$ (see also Remark~\ref{rem-2.2}).
Next, we observe that the set $\{u\in\Omega\cap \ker L \colon Mu=0\}$ is a compact subset of $\Omega$, since it is closed and contained in the compact set
$\{u\in\Omega\cap \ker L \colon QMu=0\}=\mathcal{S}_{0}$ (by condition $(iii)$).
Finally, applying Lemma~\ref{lem-2.3}, the claim follows.

From \textit{Step~1} and \textit{Step~2} the proof of the theorem is concluded.
\end{proof}

\begin{remark}\label{rem-2.3}
In Theorem~\ref{th-2.1} we define an homotopy $\tilde{M}(u,\vartheta)$ transforming only the last equation of the systems.
We underline that the same result is valid considering homotopies of the form
\begin{equation*}
\begin{cases}
\, M_{i}(u_{i+1}), & \text{for } \, i=1,\ldots,k-1; \\
\, \vartheta M_{i}(u_{1},\ldots,u_{n}), & \text{for } \, i=k,\ldots,n.
\end{cases}
\end{equation*}
Anyway, in our presentation we prefer to state the results as in Theorem~\ref{th-2.1} in order to present a version which is suitable for the application in Section~\ref{section-3}.
$\hfill\lhd$
\end{remark}

The classical Mawhin's continuation theorem (cf.~\cite[Proposition~2.1]{Ma-74}) deals with an open and \textit{bounded} set $\Omega$
such that for each $\lambda \in \mathopen{]}0,1\mathclose{[}$ the equation $Lu = \lambda Nu$ has no solutions in $\partial\Omega$
and, moreover,
\begin{equation}\label{eq-2.6}
\text{\rm deg}_{B}(-JQN|_{\Omega\cap\ker L},\Omega\cap\ker L,0)\neq0.
\end{equation}
Under these assumptions, there exists a solution $u\in \overline{\Omega}$ to the coincidence equation $Lu = Nu$.
Clearly, in the same situation (i.e.~when $\Omega$ is open and bounded), we could state an analogous existence result for system
\begin{equation*}
\begin{cases}
\, L_{i}u_{i} = M_{i}(u_{i+1}), & i=1,\ldots,n-1; \\
\, L_{n}u_{n} = M_{n}(u_{1},\ldots,u_{n}),
\end{cases}
\leqno{(\mathscr{P}_{1})}
\end{equation*}
using Theorem~\ref{th-2.1}, via the homotopy described in $(\mathscr{P}_{\vartheta})$. In such a case, we should
suppose, instead of \eqref{eq-2.6}, that
\begin{equation}\label{eq-2.7}
\text{\rm deg}_{B}(\mathcal{M},\Omega\cap\ker L,0)\neq0
\end{equation}
holds (where $\mathcal{M}$ is defined as in Theorem~\ref{th-2.1}).

Actually the above new existence result can be stated also for an open and possibly unbounded set $\Omega$.
Precisely, assuming all the hypotheses of Theorem~\ref{th-2.1} and in addition that \eqref{eq-2.7} holds, we can
immediately conclude that there exists a solution $u\in \Omega$ to the coincidence system $(\mathscr{P}_{1})$.

In order to make such new existence theorems useful for the
applications, we need first to provide more explicit conditions in
order to evaluate the Brouwer degree associated with the map
$\mathcal{M}$. Therefore, we conclude this section with a result
that allows us to compute the degree of a map having the same
structure as $\mathcal{M}$ in Theorem~\ref{th-2.1}.
To this aim, we first introduce the following notation.
Keeping for the rest of the section the hypotheses of Theorem~\ref{th-2.1}, for $i=1,\ldots,n-1$,
let us define the maps $\eta_{i}\colon\Omega_{i+1}\cap \ker L_{i+1} \to \ker L_{i}$ as
\begin{equation*}
\eta_{i}(w) := -J_{i}Q_{i}M_{i} w, \quad w\in\Omega_{i+1}\cap \ker L_{i+1},
\end{equation*}
and the map $\eta_{n}\colon\tilde{\Omega}_{1} \to \ker L_{n}$ as
\begin{equation*}
\eta_{n}(w) := -J_{n}Q_{n}M_{n}(w,0,\ldots,0), \quad w\in\tilde{\Omega}_{1},
\end{equation*}
where $\tilde{\Omega}_{1}:=\{w\in \ker L_{1} \colon (w,0,\ldots,0)\in \Omega\}$.

We will also assume some additional conditions which simplify the statement of the next result and which are natural
for the applications presented in Section~\ref{section-3} and Section~\ref{section-4}.
In more detail, in Lemma~\ref{lem-2.4} we assume this crucial hypothesis
\begin{itemize}
\item[$(h_{1})$] $\text{\rm dim}(\ker L_{i}) = d$, for all $i=1,\ldots,n$.
\end{itemize}
Accordingly, for $i=1,\ldots,n$, it is not restrictive to identify $\ker L_{i}$ with $\mathbb{R}^{d}$.
Consequently, condition $(h_{1})$ ensures that $\text{\rm dim}(\text{\rm coker}\,L_{i}) = d$, for all $i=1,\ldots,n$.
Moreover, we also identify $\text{\rm coker}\,L_{i}$ with $\mathbb{R}^{d}$, for all $i=1,\ldots,n$.
Under this position, without loss of generality, for all $i=1,\ldots,n$,
we take $J_{i}=Id_{\mathbb{R}^{d}}$ as linear orientation-preserving isomorphism from $\mathbb{R}^{d}$ to $\mathbb{R}^{d}$.
With this in mind, in the sequel, by an abuse of notation, we will write $-Q_{i}M_{i}$ in place of $-J_{i}Q_{i}M_{i}$.

Under this convention, we state the following result.

\begin{lemma}\label{lem-2.4}
Let $L$, $M$ and $\Omega$ be as in Theorem~\ref{th-2.1}. Let $\mathcal{M}\colon \Omega\cap \ker L \to \ker L$ be defined as
\begin{equation*}
\mathcal{M}(u):=\bigl{(}\eta_{1}(u_{1}),\ldots,\eta_{n-1}(u_{n}),-J_{n}Q_{n}M_{n}|_{\Omega\cap\ker L}u \bigr{)}, \quad u\in\Omega\cap\ker L.
\end{equation*}
Moreover, assume that the degree $\text{\rm deg}_{B}(\mathcal{M},\Omega \cap \ker L,0)$ is well-defined
and suppose that the following conditions hold:
\begin{itemize}
\item[$(h_{1})$] $\text{\rm dim}(\ker L_{i}) = d$, for all $i=1,\ldots,n$;
\item[$(h_{2})$] $0_{X_{i}}\in \Omega_{i}\cap\ker L_{i}$, for all $i=2,\ldots,n$;
\item[$(h_{3})$] $\{w\in\Omega_{i+1}\cap\ker L_{i+1} \colon \eta_{i}(w)=0_{X_{i}}\}=\{0_{X_{i+1}}\}$, for all $i=1,\ldots,n-1$.
\end{itemize}
Then
\begin{equation*}
\begin{aligned}
& \text{\rm deg}_{B}(\mathcal{M},\Omega \cap \ker L,0) =
\\ &= (-1)^{d(n+1)} \, \text{\rm deg}_{B}(\eta_{n},\tilde{\Omega}_{1},0) \cdot \prod_{i=1}^{n-1} \text{\rm deg}_{B}(\eta_{i},\Omega_{i+1} \cap \ker L_{i+1},0).
\end{aligned}
\end{equation*}
\end{lemma}

\begin{proof}
Let
\begin{equation*}
\tilde{\Omega} := \tilde{\Omega}_{1} \times (\Omega_{2}\cap\ker L_{2}) \times \dots \times (\Omega_{n}\cap\ker L_{n})
\end{equation*}
and let $\eta \colon\tilde{\Omega} \to \ker L$ be defined as
\begin{equation*}
\eta (u) = \bigl{(} \eta_{1}(u_{2}), \ldots, \eta_{n-1}(u_{n}), \eta_{n}(u_{1}) \bigr{)}, \quad u=(u_{1},\ldots,u_{n})\in\tilde{\Omega}.
\end{equation*}
Notice that $\tilde{\Omega}\subseteq \Omega\cap\ker L$.

\smallskip

\noindent
\textit{Step~1. } We claim that
\begin{equation*}
\text{\rm deg}_{B}(\mathcal{M},\Omega \cap \ker L,0) = \text{\rm deg}_{B}(\eta,\tilde{\Omega},0).
\end{equation*}

We introduce the operator $\tilde{\mathcal{M}} \colon \mathopen{[}0,1\mathclose{]} \times (\Omega\cap\ker L) \to \ker L$
of the form $\tilde{\mathcal{M}}=(\tilde{M}_{1},\ldots,\tilde{M}_{n})$,
with $\tilde{M}_{i} \colon \mathopen{[}0,1\mathclose{]} \times (\Omega\cap\ker L) \to \ker L_{i}$ defined as
\begin{equation*}
\begin{cases}
\, \tilde{M}_{i}(\vartheta,u) := \eta_{i}(u_{i+1}), & i=1,\ldots,n-1; \\
\, \tilde{M}_{n}(\vartheta,u) := -Q_{n}M_{n}(u_{1},\vartheta u_{2},\ldots,\vartheta u_{n}).
\end{cases}
\end{equation*}
We stress that $\tilde{\mathcal{M}}$ is a completely continuous operator.

We claim that the set
\begin{equation*}
\tilde{\mathcal{S}}:=\bigcup_{\vartheta\in\mathopen{[}0,1\mathclose{]}}\bigl{\{}u\in \Omega\cap\ker L \colon \tilde{\mathcal{M}}(\vartheta,u)=0 \bigr{\}}
\end{equation*}
is a compact subset of $\Omega\cap \ker L$.

Let us fix an arbitrary $\vartheta\in \mathopen{[}0,1\mathclose{]}$. If $u \in \Omega\cap\ker L$ is such that $\tilde{\mathcal{M}}(\vartheta,u)=0$, then
\begin{equation*}
\begin{cases}
\, \eta_{i}(u_{i+1})=0, & i=1,\ldots,n-1; \\
\, -Q_{n}M_{n}(u_{1},\vartheta u_{2},\ldots,\vartheta u_{n})=0.
\end{cases}
\end{equation*}
From the first $(n-1)$ equations and hypothesis $(h_{3})$, we immediately obtain that
\begin{equation*}
u_{i}=0_{X_{i}}, \quad \text{for all } \, i=2,\ldots,n,
\end{equation*}
and hence the last equation reads as follows
\begin{equation*}
\eta_{n}(u_{1})=-Q_{n}M_{n}(u_{1},0,\ldots,0)=0_{X_{n}}.
\end{equation*}
We conclude that
\begin{equation*}
\tilde{\mathcal{S}} = \bigl{\{}(w,0,\ldots,0)\in\Omega \colon w\in \ker L_{1} \bigr{\}} = \bigl{\{}u\in\Omega\cap \ker L \colon \mathcal{M}u=0 \bigr{\}}
\end{equation*}
is a compact subset of $\Omega\cap \ker L$ (since $\text{\rm deg}_{B}(\mathcal{M},\Omega \cap \ker L,0)$ is well-defined).

The homotopic invariance of the Brouwer degree implies that
\begin{equation*}
\vartheta \mapsto \text{\rm deg}_{B}(\tilde{\mathcal{M}}(\vartheta,\cdot),\Omega\cap \ker L,0)
\end{equation*}
is well-defined and constant on $\mathopen{[}0,1\mathclose{]}$.
In particular, since $\tilde{S}\subseteq\tilde{\Omega}$, it holds that
\begin{equation*}
\text{\rm deg}_{B}(\mathcal{M},\Omega \cap \ker L,0) = \text{\rm deg}_{B}(\tilde{\mathcal{M}}(0,\cdot),\Omega\cap \ker L,0) = \text{\rm deg}_{B}(\eta,\tilde{\Omega},0).
\end{equation*}

\smallskip

\noindent
\textit{Step~2. }
Let $\tilde{\eta} \colon \tilde{\Omega} \to \ker L$ be defined as
\begin{equation*}
\tilde{\eta} (u) = \bigl{(} \eta_{n}(u_{1}), \eta_{1}(u_{2}), \ldots, \eta_{n-1}(u_{n}) \bigr{)}, \quad u=(u_{1},\ldots,u_{n})\in\tilde{\Omega}.
\end{equation*}
Clearly
\begin{equation*}
\tilde{\eta}(u) = P \eta(u), \quad \forall \, u\in \tilde{\Omega},
\end{equation*}
where
\begin{equation*}
P =
\begin{pmatrix}
0 & 0 & \cdots & 0 & I_{d} \\
I_{d} & 0 & \cdots & 0 & 0 \\
0 & \ddots  & \ddots & \vdots & \vdots \\
\vdots & \ddots & I_{d} & 0 & 0 \\
0 & \cdots & 0 & I_{d} & 0
 \end{pmatrix}
\in \mathbb{R}^{dn\times dn}
\end{equation*}
is a permutation matrix with determinant $\det (P)=(-1)^{d(n+1)}$, where $I_{d}:=Id_{\mathbb{R}^{d}}$.
Therefore, using the definition of the Brouwer degree of a composition of maps, we obtain
\begin{equation*}
\begin{aligned}
\text{\rm deg}_{B}(\tilde{\eta},\tilde{\Omega},0) &= \text{\rm deg}_{B}(P \eta,\tilde{\Omega},0) = \text{\rm sign}(\det (P)) \, \text{\rm deg}_{B}(\eta,\tilde{\Omega},0)
\\ &= (-1)^{d(n+1)} \, \text{\rm deg}_{B}(\eta,\tilde{\Omega},0).
\end{aligned}
\end{equation*}
Now, the multiplicativity property of the Brouwer degree (cf.~\cite[Theorem~11.3]{Br-14}) gives
\begin{equation*}
\text{\rm deg}_{B}(\tilde{\eta},\tilde{\Omega},0) = \text{\rm deg}_{B}(\eta_{n},\tilde{\Omega}_{1},0) \cdot \prod_{i=1}^{n-1}  \text{\rm deg}_{B}(\eta_{i},\Omega_{i+1}\cap\ker L_{i+1},0).
\end{equation*}

\smallskip

\noindent
From \textit{Step~1} and \textit{Step~2}, we have
\begin{equation*}
\begin{aligned}
& \text{\rm deg}_{B}(\mathcal{M},\Omega \cap \ker L,0) = \text{\rm deg}_{B}(\eta,\tilde{\Omega},0) = \\
&= (-1)^{d(n+1)} \, \text{\rm deg}_{B}(\eta_{n},\tilde{\Omega}_{1},0) \cdot \prod_{i=1}^{n-1}  \text{\rm deg}_{B}(\eta_{i},\Omega_{i+1}\cap\ker L_{i+1},0)
\end{aligned}
\end{equation*}
and the lemma follows.
\end{proof}

\section{Periodic solutions to cyclic feedback systems: homotopy to the averaged nonlinearity}\label{section-3}

In this section we show an application of the theory presented in Section~\ref{section-2} to
the $T$-periodic problem (for $T>0$) associated with the differential system
\begin{equation*}
\begin{cases}
\, x_{1}' = g_{1}(x_{2}) \\
\, x_{2}' = g_{2}(x_{3}) \\
\, \quad \vdots \\
\, x_{n-1}' = g_{n-1}(x_{n}) \\
\, x_{n}' = h(t,x_{1},\ldots,x_{n}),
\end{cases}
\leqno{(\mathscr{C})}
\end{equation*}
which has been considered in the introduction.
Throughout this section, we assume that, for $i=1,\ldots,n-1$, the maps $g_{i}\colon \mathbb{R}^{m} \to \mathbb{R}^{m}$ are continuous
and $h\colon\mathopen{[}0,T\mathclose{]}\times \mathbb{R}^{m}\times \dots \times \mathbb{R}^{m} \to \mathbb{R}^{m}$
is an $L^{1}$-Carath\'{e}odory function.
A \textit{$T$-periodic solution} of $(\mathscr{C})$ is a vector function $x=(x_{1},\ldots,x_{n})$ such that,
for every $i=1,\ldots,n$, $x_{i}\colon \mathopen{[}0,T\mathclose{]} \to \mathbb{R}^{m}$
is an absolutely continuous function such that $x_{i}(0)=x_{i}(T)$
and moreover $x(t)$ satisfies $(\mathscr{C})$ for a.e.~$t\in\mathopen{[}0,T\mathclose{]}$.
It is a well-known fact that, if we suppose that $\mathbb{R} \ni t\mapsto h(t,s_{1},\ldots,s_{n})$ is a $T$-periodic map,
then any $T$-periodic solution according to our definition can be extended to the whole real line to
an absolutely continuous solution of $(\mathscr{C})$ such that $x(t+T) = x(t)$ for all $t\in\mathbb{R}$.

\begin{remark}\label{rem-3.1}
In order to simplify our presentation, we have confined ourselves to the case in which the right-hand side of system $(\mathscr{C})$
is defined on the whole space $\mathbb{R}^{mn}$.
However, the abstract results in Section~\ref{section-2} are suited to be applied also to the case in which one or more components
of the vector field in $(\mathscr{C})$ are defined only on some open subsets of the involved Euclidean spaces.
In particular, we will state Lemma~\ref{lem-cycl} and the subsequent results by assuming $\Omega \subseteq \text{\rm dom}\,M$,
where $M$ will be the Nemytskii operator associated to the right-hand side of $(\mathscr{C})$, so that they are applicable to
the most general situation. Clearly, in our simplified setting the hypothesis $\Omega \subseteq \text{\rm dom}\,M$
will be equivalent to consider as $\Omega$ just an open subset of $\mathcal{C}(\mathopen{[}0,T\mathclose{]},\mathbb{R}^{mn})$.
$\hfill\lhd$
\end{remark}

\medskip

In order to enter the setting presented in Section~\ref{section-2} and to write
the system in the form
\begin{equation*}
Lx = Mx,\quad x\in \text{\rm dom}\,L\cap \text{\rm dom}\,M,
\end{equation*}
we will adapt to our situation the classical treatment in \cite{Ma-79}.
For $i=1,\ldots,n$, let $X_{i}:= \mathcal{C}(\mathopen{[}0,T\mathclose{]},\mathbb{R}^{m})$ be the space of continuous functions $x_{i} \colon \mathopen{[}0,T\mathclose{]} \to \mathbb{R}^{m}$,
endowed with the $\sup$-norm
\begin{equation*}
\|x_{i}\|_{\infty} := \max_{t\in \mathopen{[}0,T\mathclose{]}} |x_{i}(t)|,
\end{equation*}
and let $Z_{i}:=L^{1}(\mathopen{[}0,T\mathclose{]},\mathbb{R}^{m})$ be the space of integrable functions $z_{i} \colon \mathopen{[}0,T\mathclose{]} \to \mathbb{R}^{m}$,
endowed with the norm
\begin{equation*}
\|z_{i}\|_{L^{1}}:= \int_{0}^{T} |z_{i}(t)|~\!dt.
\end{equation*}
In this manner, we have $X = \mathcal{C}(\mathopen{[}0,T\mathclose{]},\mathbb{R}^{mn})$ and $Z =L^{1}(\mathopen{[}0,T\mathclose{]},\mathbb{R}^{mn})$ (with the standard norms).

For $i=1,\ldots,n$, we consider the linear differential operator $L_{i}\colon \text{\rm dom}\,L_{i} \to Z_{i}$ defined as
\begin{equation*}
(L_{i}x_{i})(t):= x'_{i}(t), \quad t\in\mathopen{[}0,T\mathclose{]},
\end{equation*}
where $\text{\rm dom}\,L_{i}$ is determined by the functions of $X_{i}$
which are absolutely continuous and satisfy the periodic boundary condition
\begin{equation}\label{per-cond}
x_{i}(0) = x_{i}(T).
\end{equation}
Therefore, $L_{i}$ is a Fredholm map of index zero, $\ker L_{i}$ and $\text{\rm coker}\,L_{i}$ are made up of the constant functions in $\mathbb{R}^{m}$ and
\begin{equation*}
\text{\rm Im}\,L_{i} = \biggl{\{} z_{i}\in Z_{i} \colon \int_{0}^{T} z_{i}(t)~\!dt = 0 \biggr{\}}.
\end{equation*}
As projectors $P_{i} \colon X_{i} \to \ker L_{i}$ and $Q_{i} \colon Z_{i} \to \text{\rm coker}\,L_{i}$ associated with $L_{i}$, for $i=1,\ldots,n$, we choose the average operators
\begin{equation*}
P_{i}x_{i} = Q_{i}x_{i} := \dfrac{1}{T}\int_{0}^{T} x_{i}(t)~\!dt.
\end{equation*}
Notice that $\ker P_{i}$ is given by the continuous functions with mean value zero.
Next, let $K_{i} \colon \text{\rm Im}\,L_{i} \to \text{\rm dom}\,L_{i} \cap \ker P_{i}$ be the right inverse of $L_{i}$,
which is the operator that to any function $z_{i}\in Z_{i}$ with $\int_{0}^{T} z_{i}(t)~\!dt =0$ associates the unique solution $x_{i}(t)$ of
\begin{equation*}
x_{i}'= z_{i}(t), \quad \text{ with } \; \int_{0}^{T} x_{i}(t)~\!dt = 0,
\end{equation*}
which clearly satisfies the boundary condition \eqref{per-cond}. Finally, we take the identity map in $\mathbb{R}^{m}$ as
linear orientation-preserving isomorphism $J_{i} \colon \text{\rm coker}\,L_{i} \to \ker L_{i}$.

Now we consider as nonlinear operator $M \colon \text{\rm dom}\,M = X\to Z$
the Nemytskii operator induced by the functions $g_{i}$ and $h$, namely
the operator $M$ has the following components
\begin{equation*}
\begin{cases}
\, M_{i}(x_{i+1})(t):= g_{i}(x_{i+1}(t)), & i=1,\ldots,n-1; \\
\, M_{n}(x_{1},\ldots,x_{n})(t) := h(t,x_{1}(t),\ldots,x_{n}(t)),
\end{cases}
\end{equation*}
where $x=(x_{1},\ldots,x_{n})\in \text{\rm dom}\,M$ and $t\in\mathopen{[}0,T\mathclose{]}$.
From the above hypotheses it follows that $M$ is an $L$-completely continuous operator.

Finally, we introduce the averaged vector field
$h^{\#} \colon \mathbb{R}^{mn} \to \mathbb{R}^{m}$ defined by
\begin{equation*}
h^{\#}(s):= \dfrac{1}{T} \int_{0}^{T} h(t,s_{1},\ldots,s_{n})~\!dt, \quad s=(s_{1},\ldots,s_{n})\in\mathbb{R}^{mn}.
\end{equation*}
We also define $\hat{g}\colon \mathbb{R}^{mn} \to \mathbb{R}^{mn}$ as
\begin{equation*}
\hat{g}(s):=\bigl{(}g_{1}(s_{2}),\ldots,g_{n-1}(s_{n}),h^{\#}(s)\bigr{)}, \quad s=(s_{1},\ldots,s_{n})\in\mathbb{R}^{mn}.
\end{equation*}
Note that, according to the above positions, it turns out that $\hat{g} = JQM|_{\ker L}$.

\medskip

We are now in position to state our first result, which is a direct consequence of Theorem~\ref{th-2.1}.

\begin{lemma}\label{lem-cycl}
Let $\Omega \subseteq \text{\rm dom}\,M$ be an open (possibly unbounded) set.
Suppose that the following conditions hold.
\begin{itemize}
\item[$(c_{1})$]
There exists a compact set $\mathcal{K}\subseteq \Omega$ containing all the possible $T$-periodic solutions of
\begin{equation*}
\begin{cases}
\, x_{1}' = g_{1}(x_{2}) \\
\, x_{2}' = g_{2}(x_{3}) \\
\, \quad \vdots \\
\, x_{n-1}' = g_{n-1}(x_{n}) \\
\, x_{n}' = \vartheta h(t,x_{1},\ldots,x_{n}),
\end{cases}
\leqno{(\mathscr{C}_{\vartheta})}
\end{equation*}
for any $\vartheta\in\mathopen{]}0,1\mathclose{]}$.
\item[$(c_{2})$]
The set $\hat{g}^{-1}(0)\cap \Omega$
is compact.
\end{itemize}
Then
\begin{equation*}
D_{L}(L-M,\Omega) = (-1)^{mn} \, \text{\rm deg}_{B}(\hat{g},\Omega \cap \mathbb{R}^{mn},0).
\end{equation*}
\end{lemma}

\begin{proof}
According to the above positions, we have
\begin{equation*}
-J_{i}Q_{i}M_{i} (s_{i+1})= -g_{i} (s_{i+1}), \quad \forall \, s_{i+1}\in\Omega_{i+1}\cap \ker L_{i+1}, \quad \forall \, i=1,\ldots,n-1,
\end{equation*}
and
\begin{equation*}
-J_{n}Q_{n}M_{n}(s)= -h^{\#}(s), \quad \forall \, s\in\Omega\cap \ker L.
\end{equation*}
In order to apply Theorem~\ref{th-2.1}, we have to verify that hypotheses $(i)$, $(ii)$ and $(iii)$ of that theorem are satisfied.
Clearly $(i)$ holds, since the only constant function with zero mean value is the null function.
On the other hand, $(ii)$ and $(iii)$ are direct consequences of $(c_{1})$ and $(c_{2})$, respectively, in the
functional analytic setting that we have introduced at the beginning of the section. Then
\begin{equation*}
D_{L}(L-M,\Omega) = \text{\rm deg}_{B}(-\hat{g},\Omega \cap \mathbb{R}^{mn},0)
\end{equation*}
and thus the conclusion follows.
\end{proof}

From Lemma~\ref{lem-cycl} we immediately obtain the following existence result. The obvious proof is omitted.

\begin{theorem}\label{th-cycl1}
Let $\Omega \subseteq \text{\rm dom}\,M$ be an open (possibly unbounded) set. Suppose that $(c_{1})$ and $(c_{2})$ hold.
If
\begin{equation*}
\text{\rm deg}_{B}(\hat{g},\Omega \cap \mathbb{R}^{mn},0)\neq 0,
\end{equation*}
then there exists at least a $T$-periodic solution of $(\mathscr{C})$ in $\Omega$.
\end{theorem}

Theorem~\ref{th-cycl1} relies on the evaluation of the Brouwer degree
of $\hat{g}$. We show now how to compute this degree in terms of that relative to $h^{\#}$, via
Lemma~\ref{lem-2.4}. To this end, given an open set $\Omega \subseteq \text{\rm dom}\,M$ we recall the definition
of $\Omega_{i}:= \pi_{i}^{X}(\Omega)$.
We also set
\begin{equation*}
\mathcal{O}_{1} := \bigl{\{}\omega\in\mathbb{R}^{m}\colon(\omega,0,\ldots,0)\in\Omega\bigr{\}}
\end{equation*}
(which corresponds to the set $\tilde{\Omega}_{1}$ of Lemma~\ref{lem-2.4})
and
\begin{equation*}
\mathcal{O}_{i}:= \Omega_{i}\cap \mathbb{R}^{m}, \quad \text{for } \, i=2,\ldots,n.
\end{equation*}
Finally, we introduce the map $h^{*}\colon\mathbb{R}^{m}\to\mathbb{R}^{m}$
\begin{equation*}
h^{*}(\omega):= h^{\#}(\omega,0,\ldots,0), \quad \omega\in \mathbb{R}^{m}.
\end{equation*}
Then, we have the following result.

\begin{proposition}\label{prop-3.1}
Let $\Omega \subseteq \text{\rm dom}\,M$ be an open (possibly unbounded) set.
Suppose that
\begin{itemize}
\item [$(c_{3})$] $0 \in \mathcal{O}_{i}$, for each $i=2,\ldots,n$;
\item [$(c_{4})$] $g_{i}(0) = 0$, for each $i=1,\ldots,n-1$;
\item [$(c_{5})$] $g_{i}(\omega) \neq 0$ for every $\omega \in
\mathcal{O}_{i+1}\setminus\{0\}$, for each $i=1,\ldots,n-1$;
\item [$(c_{6})$] the set $(h^{*})^{-1}(0)\cap\mathcal{O}_{1}$ is compact.
\end{itemize}
Then, $\text{\rm deg}_{B}(\hat{g},\Omega \cap \mathbb{R}^{mn},0)$ is well-defined
and the following formula holds
\begin{equation*}
\text{\rm deg}_{B}(\hat{g},\Omega \cap \mathbb{R}^{mn},0) =
(-1)^{m(n+1)} \,\text{\rm deg}_{B}(h^{*},\mathcal{O}_{1},0) \cdot \prod_{i=1}^{n-1}\text{\rm deg}_{B}(g_{i},\mathcal{O}_{i+1},0).
\end{equation*}
\end{proposition}

\begin{proof}
As a first step, we observe that $\text{\rm deg}_{B}(\hat{g},\Omega \cap \mathbb{R}^{mn},0)$ is well-defined if condition
$(c_{2})$ holds. Now, using  $(c_{3})$, $(c_{4})$ and $(c_{5})$, we immediately deduce
\begin{equation*}
\hat{g}^{-1}(0) \cap \Omega
= \bigl{\{}(\omega,0,\ldots,0) \in\Omega\cap \mathbb{R}^{mn}\colon \omega\in(h^{*})^{-1}(0)\bigr{\}}.
\end{equation*}
Hence $(c_{2})$ is valid if and only if $(c_{6})$ is satisfied.

As a second step, we assume that $\text{\rm deg}_{B}(\hat{g},\Omega \cap \mathbb{R}^{mn},0)$ is well-defined.
Using the above positions, it is straightforward to check that the assumptions of Lemma~\ref{lem-2.4}
are satisfied with dimension $d = m$ in $(h_{1})$, $(h_{2})$ and $(h_{3})$
following from $(c_{3})$ and from $(c_{4})$, $(c_{5})$, respectively.
\end{proof}

From now on, in the next results, we will assume all the hypotheses of Proposition~\ref{prop-3.1}. In this manner,
condition $(c_{6})$ will ensure that both the degrees
$\text{\rm deg}_{B}(\hat{g},\Omega \cap \mathbb{R}^{mn},0)$ and $\text{\rm deg}_{B}(h^{*},\mathcal{O}_{1},0)$
are well-defined.

Combining Theorem~\ref{th-cycl1} and Proposition~\ref{prop-3.1}, we obtain the following.

\begin{corollary}\label{cor-3.1}
Let $\Omega \subseteq \text{\rm dom}\,M$ be an open (possibly unbounded) set.
Assume $(c_{1})$, $(c_{3})$, $(c_{4})$, $(c_{5})$ and $(c_{6})$.
If
\begin{equation*}
\text{\rm deg}_{B}(g_{i},\mathcal{O}_{i+1},0)\neq0, \quad \text{for all } \, i=1,\ldots,n-1,
\end{equation*}
and
\begin{equation*}
\text{\rm deg}_{B}(h^{*},\mathcal{O}_{1},0)\neq0,
\end{equation*}
then there exists at least a $T$-periodic solution of $(\mathscr{C})$ in $\Omega$.
\end{corollary}

The above corollary can be further simplified if we assume the following hypothesis which is rather
natural in our framework:

\begin{itemize}
\item [$(c_{*})$]
$0 \in \mathcal{O}_{i+1}$ and  $g_{i}|_{\mathcal{O}_{i+1}} \colon \mathcal{O}_{i+1} \to g_{i}(\mathcal{O}_{i+1})\subseteq\mathbb{R}^{m}$
is a homeomorphism with $g_{i}(0) =0$, for all $i=1,\ldots,n-1$.
\end{itemize}
Notice that $(c_{*})$ implies $(c_{3})$, $(c_{4})$, $(c_{5})$ and, moreover, for every $i=1,\ldots,n-1$,
$\text{\rm deg}_{B}(g_{i},\mathcal{O}_{i+1},0) = \pm1$
(the sign depending on the fact that $g_{i}$ is an orientation-preserving or orientation-reversing homeomorphism).
As a consequence, the following result holds.

\begin{corollary}\label{cor-3.2}
Let $\Omega \subseteq \text{\rm dom}\,M$ be an open (possibly unbounded) set.
Assume $(c_{1})$, $(c_{6})$ and $(c_{*})$.
If
\begin{equation*}
\text{\rm deg}_{B}(h^{*},\mathcal{O}_{1},0)\neq0,
\end{equation*}
then there exists at least a $T$-periodic solution of $(\mathscr{C})$ in $\Omega$.
\end{corollary}

\medskip

From now on we deal with an open and \textit{bounded} set $\Omega$
with $\overline{\Omega} \subseteq \text{\rm dom}\,M$. In order to
present the previous results in this special case, we need to slightly modify some of the hypotheses previously introduced. We
will only state the results, omitting the  proofs which require only obvious changes in the previous arguments.

The following theorem is a continuation result which is a variant of Theorem~\ref{th-cycl1}.

\begin{theorem}\label{th-cycl2}
Let $\Omega$ be an open and bounded set with $\overline{\Omega} \subseteq \text{\rm dom}\,M$.
Suppose that the following conditions hold.
\begin{itemize}
\item[$(c'_{1})$]
For each $\vartheta\in\mathopen{]}0,1\mathclose{[}$ there is no
$T$-periodic solution of $(\mathscr{C}_{\vartheta})$ with $x\in \partial\Omega$.
\item[$(c'_{2})$]
$\hat{g}^{-1}(0)\cap \partial\Omega = \emptyset$.
\end{itemize}
If
\begin{equation*}
\text{\rm deg}_{B}(\hat{g},\Omega \cap \mathbb{R}^{mn},0)\neq 0,
\end{equation*}
then there exists at least a $T$-periodic solution of $(\mathscr{C})$ in $\overline{\Omega}$.
\end{theorem}

We underline that Proposition~\ref{prop-3.1} is still valid in this special framework by replacing hypothesis $(c_{5})$ and $(c_{6})$
with the following ones, respectively.
\begin{itemize}
\item [$(c'_{5})$] $g_{i}(\omega) \neq 0$ for every $\omega \in \overline{\mathcal{O}_{i+1}}\setminus\{0\}$, for each $i=1,\ldots,n-1$.
\item [$(c'_{6})$] $(h^{*})^{-1}(0)\cap \partial\mathcal{O}_{1} = \emptyset$.
\end{itemize}
In particular, we recall that condition $(c'_{6})$ guarantees that $\text{\rm deg}_{B}(h^{*},\mathcal{O}_{1},0)$ is well-defined.

Then, from Theorem~\ref{th-cycl2}, together with the modification of Proposition~\ref{prop-3.1} described above, we have the next result (analogous to Corollary~\ref{cor-3.1})

\begin{corollary}\label{cor-3.3}
Let $\Omega$ be an open and bounded set with $\overline{\Omega} \subseteq \text{\rm dom}\,M$.
Assume $(c'_{1})$, $(c_{3})$, $(c_{4})$, $(c'_{5})$ and $(c'_{6})$.
If
\begin{equation*}
\text{\rm deg}_{B}(g_{i},\mathcal{O}_{i+1},0)\neq0, \quad \text{for all } \, i=1,\ldots,n-1,
\end{equation*}
and
\begin{equation*}
\text{\rm deg}_{B}(h^{*},\mathcal{O}_{1},0)\neq0,
\end{equation*}
then there exists at least a $T$-periodic solution of $(\mathscr{C})$ in $\overline{\Omega}$.
\end{corollary}

In the same spirit of Corollary~\ref{cor-3.2}, introducing the hypothesis
\begin{itemize}
\item [$(c'_{*})$]
$0 \in \mathcal{O}_{i+1}$ and  $g_{i}|_{\overline{\mathcal{O}_{i+1}}}
\colon \overline{\mathcal{O}_{i+1}}\to g_{i}(\overline{\mathcal{O}_{i+1}})\subseteq \mathbb{R}^{m}$
is a homeomorphism with $g_{i}(0) =0$, for all $i=1,\ldots,n-1$,
\end{itemize}
we can state the following.

\begin{corollary}\label{cor-3.4}
Let $\Omega$ be an open and bounded set with $\overline{\Omega} \subseteq \text{\rm dom}\,M$.
Assume $(c'_{1})$, $(c'_{6})$ and $(c'_{*})$.
If
\begin{equation*}
\text{\rm deg}_{B}(h^{*},\mathcal{O}_{1},0)\neq0,
\end{equation*}
then there exists at least a $T$-periodic solution of $(\mathscr{C})$ in $\overline{\Omega}$.
\end{corollary}

\medskip

We conclude this section by showing a possible application where the hypothesis $(c'_{*})$ is automatically satisfied.
To this end, we consider the periodic problem associated with the $n$-th order differential system
for $u(t)\in \mathbb{R}^{m}$
\begin{equation}\label{eq-3.1}
\bigl{(}\varphi_{n-1}\bigl{(}\bigl{(}\ldots\bigl{(}\varphi_{2}\bigl{(}\bigl{(}\varphi_{1}(u')\bigr{)}'\bigr{)}\bigr{)}'\ldots\bigr{)}'\bigr{)}\bigr{)}' + k\bigl{(}t,u,u',\ldots,u^{(n)}\bigr{)} = 0,
\end{equation}
where, for each $i=1,\ldots,n-1$, $\varphi_{i}\colon \mathbb{R}^{m} \to \mathbb{R}^{m}$ is a homeomorphism with $\varphi_{i}(0)=0$.
Our study generalizes previous investigations in \cite{CaQiZa-99} in the scalar case ($m=1$).
Equation \eqref{eq-3.1} can be equivalently written as a cyclic feedback type system in $\mathbb{R}^{mn}$
of the form
\begin{equation}\label{syst-cycl}
\begin{cases}
\, x_{1}' = \varphi_{1}^{-1}(x_{2}) \\
\, x_{2}' = \varphi_{2}^{-1}(x_{3}) \\
\, \quad \vdots \\
\, x_{n-1}' = \varphi_{n-1}^{-1}(x_{n}) \\
\, x_{n}' = h(t,x_{1},\ldots,x_{n}),
\end{cases}
\end{equation}
where
\begin{equation*}
h(t,s_{1},s_{2},\ldots,s_{n}) := -k(t,s_{1},\varphi^{-1}_{1}(s_{2}),\ldots,\varphi_{n-1}^{-1}(s_{n})).
\end{equation*}
Observe that $h(t,s_{1},0,\ldots,0) = -k(t,s_{1},0,\ldots,0)$ and hence
\begin{equation*}
h^{*}(\omega) = -\frac{1}{T}\int_{0}^{T}k(t,\omega,0,\ldots,0)~\!dt.
\end{equation*}
From Corollary~\ref{cor-3.4} we directly obtain the following result
(the definition of the open sets $\mathcal{O}_i$ is the same as above).
The obvious proof is omitted.

\begin{theorem}\label{th-3.3}
Let $\Omega \subseteq \mathcal{C}(\mathopen{[}0,T\mathclose{]},\mathbb{R}^{mn})$ be an open and bounded set
such that $0\in \mathcal{O}_{i}$ for all $i=2,\ldots,n$.
Suppose that
\begin{itemize}
\item for each $\vartheta\in\mathopen{]}0,1\mathclose{[}$ there is no
$T$-periodic solutions of
\begin{equation*}
\begin{cases}
\, x_{1}' = \varphi_{1}^{-1}(x_{2}) \\
\, x_{2}' = \varphi_{2}^{-1}(x_{3}) \\
\, \quad \vdots \\
\, x_{n-1}' = \varphi_{n-1}^{-1}(x_{n}) \\
\, x_{n}' = \vartheta h(t,x_{1},\ldots,x_{n})
\end{cases}
\end{equation*}
with $x\in \partial\Omega$;
\item $h^{*}(\omega) \neq 0$, for every $\omega\in \partial\mathcal{O}_{1}$ and $\text{\rm deg}_{B}(h^{*},\mathcal{O}_{1},0)\neq 0$.
\end{itemize}
Then there exists at least a $T$-periodic solution $x(t)$ of \eqref{syst-cycl} in $\overline{\Omega}$.
\end{theorem}

An important case of system \eqref{eq-3.1} is given by the second order $\phi$-Laplacian equation
\begin{equation}\label{eq-phik}
\bigl{(}\phi(u')\bigr{)}' + k(t,u,u') = 0,
\end{equation}
where $\phi \colon \mathbb{R}^{m}\to \mathbb{R}^{m}$ is a homeomorphism with $\phi(0) = 0$
and $k \colon \mathopen{[}0,T\mathclose{]}\times \mathbb{R}^{m}\times \mathbb{R}^{m} \to \mathbb{R}^{m}$
is an $L^{1}$-Carath\'{e}odory function. System \eqref{eq-phik} plays an important role in several mathematical models
and therefore our next goal is to get some applications to this class of systems. With this respect,
it will be convenient to introduce the following notation.
We denote by $\mathcal{C}^{1}_{T}$ the space of continuously differentiable functions
$u \colon \mathopen{[}0,T\mathclose{]} \to \mathbb{R}^{m}$
satisfying the boundary condition
\begin{equation}\label{2BC}
u(0) = u(T), \quad u'(0) = u'(T).
\end{equation}
In this space, we take as a norm
\begin{equation*}
\|u\|_{\mathcal{C}^{1}}:= \max \bigl{\{} \|u\|_{\infty}, \|u'\|_{\infty} \bigr{\}},
\end{equation*}
which is equivalent to the more standard norm $\|u\|_{\infty} + \|u'\|_{\infty}$.

We also set
\begin{equation*}
k^{*}(\omega) := \frac{1}{T}\int_{0}^{T}k(t,\omega,0)~\!dt, \quad \omega\in\mathbb{R}^{m}.
\end{equation*}

As mentioned in the introduction, a relevant continuation theorem for system \eqref{eq-phik},
involving the homotopic equation
\begin{equation}\label{eq-phiklam}
\bigl{(}\phi(u')\bigr{)}' + \lambda k(t,u,u') = 0,
\end{equation}
was achieved by
Man\'{a}sevich and Mawhin in \cite{MaMa-98} under some additional hypotheses on the homeomorphism $\phi$.
In this setting, one could observe that equation \eqref{eq-phik} is equivalent to the
first order cyclic system in $\mathbb{R}^{2m}$
\begin{equation}\label{syst-cycl2}
\begin{cases}
\, x_{1}' = \phi^{-1}(x_{2}) \\
\, x'_{2} = h(t,x_{1},x_{2}),
\end{cases}
\end{equation}
where $h(t,x_{1},x_{2}):= -k(t,x_{1},\phi^{-1}(x_{2}))$ and, analogously, \eqref{eq-phiklam} can be written as
\begin{equation}\label{syst-cycl2lam}
\begin{cases}
\, x_{1}' = \phi^{-1}(x_{2}) \\
\, x'_{2} = \lambda h(t,x_{1},x_{2}),
\end{cases}
\end{equation}
so that the continuation theorem \cite[Theorem~3.1]{MaMa-98} could be derived as a corollary of
Theorem~\ref{th-3.3}, without any further condition on $\phi$. However, a deeper inspection shows that the
situation is not so simple. Indeed, in \cite{MaMa-98} the condition on the homotopic equation \eqref{eq-phiklam}
requires no solutions on the boundary of an open and bounded set in the $\mathcal{C}^{1}_{T}$-norm. Due to the fact that
this norm is strictly finer that the $\sup$-norm that we consider for our approach, it seems not obvious how to
include the results in \cite{MaMa-98} in our setting (except for very special cases of $\Omega$).
We propose below a possible way to overcome this difficulty
and thus recover Man\'{a}sevich-Mawhin continuation theorem \cite[Theorem~3.1]{MaMa-98},
without additional hypotheses on $\phi$.

\begin{theorem}\label{th-mama1}
Let $\mathcal{U}$ be an open and bounded set in $\mathcal{C}^{1}_{T}$ such that the following conditions hold.
\begin{itemize}
\item For each $\lambda\in\mathopen{]}0,1\mathclose{[}$ the problem
\begin{equation*}
\bigl{(}\phi(u')\bigr{)}' + \lambda \, k(t,u,u') = 0, \quad u(0) = u(T), \quad u'(0) = u'(T),
\end{equation*}
has no solution on $\partial\mathcal{U}$.
\item The equation $k^{*}(\omega) = 0$ has no solution on $\partial\mathcal{U}\cap \mathbb{R}^{m}$ and
\begin{equation*}
\text{\rm deg}_{B}(k^{*},\mathcal{U}\cap \mathbb{R}^{m},0)\neq0.
\end{equation*}
\end{itemize}
Then, problem \eqref{eq-phik}-\eqref{2BC} has at least a solution in $\overline{\mathcal{U}}$.
\end{theorem}

\begin{proof}
If there exists a solution in $\partial\mathcal{U}$, we are done.
Then, for the rest of the proof, we assume that problem \eqref{eq-phik}-\eqref{2BC} has no solution in $\partial\mathcal{U}$.
We split our argument into three steps.

\smallskip

\noindent
\textit{Step~1. Compactness. }
We claim that the set
\begin{equation*}
\mathcal{K} := \bigcup_{\lambda\in\mathopen{]}0,1\mathclose{]}}\Bigl{\{}u\in\overline{\mathcal{U}}\colon \bigl{(}\phi(u')\bigr{)}' + \lambda \, k(t,u,u') = 0 \Bigr{\}}
\end{equation*}
is a compact subset of $\mathcal{U}$.
To this end, let $(\lambda_{n},u_{n})\in\mathopen{]}0,1\mathclose{]}\times\overline{\mathcal{U}}$ be such that
\begin{equation*}
\bigl{(}\phi(u_{n}')\bigr{)}' + \lambda_{n} \, k(t,u_{n},u_{n}') = 0.
\end{equation*}
By assumption, $\overline{\mathcal{U}}$ is bounded, therefore there is a constant $r > 0$ such that
$\|u\|_{\infty} \leq r$ and $\|u'\|_{\infty} \leq r$, for each $u\in \overline{\mathcal{U}}$.
Then, from the Carath\'{e}odory conditions we deduce that there exists a measurable function $\rho\in L^{1}(\mathopen{[}0,T\mathclose{]})$
such that $\|k(t,u_{n}(t),u_{n}'(t))\|_{\mathbb{R}^{m}}\leq \rho(t)$, for a.e.~$t\in \mathopen{[}0,T\mathclose{]}$.
We also introduce the uniformly continuous function $\mathcal{R}(t):=\int_{0}^{t}\rho(\xi)~\!d\xi$, for $t\in \mathopen{[}0,T\mathclose{]}$.
The sequence $(u'_{n})_{n}$ is equicontinuous. Indeed,
\begin{equation*}
u'_{n}(t) = \phi^{-1}(v_{n}(t)),\quad \text{where } \, v_{n}(t):= \phi(u'_{n}(0)) -\lambda_{n} \,\int_{0}^{t} k(\xi,u_{n}(\xi),u_{n}'(\xi))~\!d\xi,
\end{equation*}
and $v_{n}$ is uniformly bounded on $\mathopen{[}0,T\mathclose{]}$ by the constant
\begin{equation*}
r_{1}:= \max \bigl{\{}\|\phi(z)\|_{\mathbb{R}^{m}} \colon \|z\|_{\mathbb{R}^{m}} \leq r \bigr{\}} + \mathcal{R}(T).
\end{equation*}
The uniform continuity of the map $\phi^{-1} \colon \mathbb{R}^{m}\to \mathbb{R}^{m}$ restricted to the closed ball $B[0,r_{1}]\subseteq \mathbb{R}^{m}$
implies that for each $\varepsilon > 0$ there is a $\delta=\delta_{\varepsilon} > 0$ such that
$\|\phi^{-1}(z) - \phi^{-1}(y)\|_{\mathbb{R}^{m}} < \varepsilon$
for all $z,y\in B[0,r_{1}]$ with $\|z - y\|_{\mathbb{R}^{m}} < \delta$.
On the other hand, given $\delta > 0$ there is $\eta=\eta_{\delta} > 0$ such that $|\mathcal{R}(t) - \mathcal{R}(s)| < \delta$
for all $t,s\in\mathopen{[}0,T\mathclose{]}$ with $|t-s|< \eta$.
Thus, given $\varepsilon > 0$, we have that
\begin{equation*}
\|u'_{n}(t) - u'_{n}(s)\|_{\mathbb{R}^{m}} = \|\phi^{-1}(v_{n}(t)) - \phi^{-1}(v_{n}(s))\|_{\mathbb{R}^{m}} < \varepsilon
\end{equation*}
whenever $\|v_{n}(t) - v_{n}(s)\|_{\mathbb{R}^{m}} <\delta$.
On the other hand,
\begin{equation*}
\begin{aligned}
\|v_{n}(t) - v_{n}(s)\|_{\mathbb{R}^{m}}
&= \biggl{\|}\lambda_{n} \, \int_s^t k(\xi,u_{n}(\xi),u_{n}'(\xi))~\!d\xi \biggr{\|}_{\mathbb{R}^{m}} \\
& \leq \biggl{|} \int_s^t \|k(\xi,u_{n}(\xi),u_{n}'(\xi))\|_{\mathbb{R}^{m}}~\!d\xi \biggr{|} \\
& \leq \biggl{|} \int_s^t \rho(\xi)~\!d\xi \biggr{|} = |\mathcal{R}(t) - \mathcal{R}(s) |.
\end{aligned}
\end{equation*}
Thus we conclude that $\|u'_{n}(t) - u'_{n}(s)\|_{\mathbb{R}^{m}} < \varepsilon$
for $|t-s|< \eta$, for every $n$.
The Ascoli-Arzel\`{a} theorem guarantees that, up to a subsequence, $u_{n}\to \tilde{u}\in \overline{\mathcal{U}}$ in the $\mathcal{C}^{1}_{T}$-norm
and we have also $\lambda_{n}\to \tilde{\lambda}\in \mathopen{[}0,1\mathclose{]}$. By the
assumption of no solutions on the boundary, we know that if $\tilde{\lambda} \in \mathopen{]}0,1\mathclose{]}$ we must have $\tilde{u}\in \mathcal{U}$.
We study now separately the case in which $\tilde{\lambda} = 0$. In this case, by the same computations as above and the
dominated convergence theorem, we find that $\phi(\tilde{u}'(t))= \phi(\tilde{u}'(0))$ for all $t\in \mathopen{[}0,T\mathclose{]}$, so that
(recalling that $\tilde{u}\in \overline{\mathcal{U}}$ satisfies \eqref{2BC}), $\tilde{u}$
is constant, that is $\tilde{u}(t) =\tilde{\omega}$ for some $\tilde{\omega} \in \overline{\mathcal{U}}\cap\mathbb{R}^{m}$.
Finally, the second hypothesis in the theorem ensures that $\tilde{\omega}\in \mathcal{U}$. The claim is proved.

\smallskip

\noindent
\textit{Step~2. A special case for the domain. }
Suppose that there exist two open bounded sets $\mathcal{U}_{1},\mathcal{U}_{2}\in\mathcal{C}(\mathopen{[}0,T\mathclose{]},\mathbb{R}^{m})$ with $0\in\mathcal{U}_{2}$ such that
\begin{equation*}
\mathcal{U} = \bigl{\{}u\in\mathcal{C}^{1}_{T} \colon u\in\mathcal{U}_{1}, \, u'\in\mathcal{U}_{2} \bigr{\}}.
\end{equation*}
We write \eqref{eq-phik} as an equivalent first order cyclic system in $\mathbb{R}^{2m}$ of the form \eqref{syst-cycl2}
with $h(t,x_{1},x_{2}):= -k(t,x_{1},\phi^{-1}(x_{2}))$.

In the Banach space $X:= \mathcal{C}(\mathopen{[}0,T\mathclose{]},\mathbb{R}^{2m})$ we define the set
\begin{equation*}
\Omega := \mathcal{U}_{1} \times \phi(\mathcal{U}_{2}) = \bigl{\{} x=(x_{1},x_{2})\in X \colon x_{1} \in \mathcal{U}_{1}, \, \phi^{-1}(x_{2}) \in \mathcal{U}_{2} \bigr{\}}.
\end{equation*}
Clearly the set $\Omega$ is open and bounded in $X$.

With these positions, we can easily check that the first hypothesis of the theorem implies that system
\eqref{syst-cycl2lam}
has no $T$-periodic solution $x\in \partial\Omega$, for any $\lambda\in \mathopen{]}0,1\mathclose{[}$,
actually for any $\lambda\in \mathopen{]}0,1\mathclose{]}$, because we have started the proof by assuming that \eqref{eq-phik}
has no solution on the boundary.
We are therefore in the setting of Theorem~\ref{th-3.3} with its first condition satisfied. Also the second condition in
Theorem~\ref{th-3.3} holds, because it follows directly from the second hypothesis of the present theorem.
Then, we can apply Theorem~\ref{th-3.3} and we obtain that
there exists at least a $T$-periodic solution $x(t)=(x_{1}(t),x_{2}(t))$ of \eqref{syst-cycl} in $\overline{\Omega}$.
Actually, we have $x\in \Omega$ (since we have started our proof by assuming that there are no solutions on the boundary).
Defining $u(t):=x_{1}(t)$ for $t\in\mathopen{[}0,T\mathclose{]}$,
we immediately conclude that $u=x_{1}\in \mathcal{U}_{1}$, $u'=\phi^{-1}(x_{2}) \in \mathcal{U}_{2}$,
then $u \in \mathcal{U}$, and $u(t)$ satisfies \eqref{eq-phik} and \eqref{2BC}.

\smallskip

\noindent
\textit{Step~3. General case. }
Let $\mathcal{U}\in\mathcal{C}^{1}_{T}$ be an open and bounded set.
From \textit{Step~1}, $\mathcal{K}$ is a compact subset of $\mathcal{U}$.
Therefore, for each point $w\in \mathcal{K}$ there is an open ball $B(w,r_{w})\subseteq \mathcal{U}$,
in the $\mathcal{C}^{1}_{T}$-norm, which is a set
of the product form as the one in \textit{Step~2}.
Indeed, $u\in B(w,r_{w})$ if and only if $\|u-w\|_{\infty}< r_{w}$ and $\|u'-w'\|_{\infty}< r_{w}$.
By a standard compactness argument, we have
\begin{equation*}
\mathcal{K} \subseteq  \bigcup_{\alpha=1}^{\ell} \mathcal{U}^{\alpha},
\end{equation*}
with $\mathcal{U}^{\alpha}\subseteq \mathcal{U}$ an open (and bounded) set of the form
\begin{equation}\label{sets-u}
\mathcal{U}^{\alpha} := \bigl{\{}u\in\mathcal{C}^{1}_{T} \colon u\in\mathcal{U}^{\alpha}_{1}, \, u'\in\mathcal{U}^{\alpha}_{2} \bigr{\}},
\end{equation}
where $\mathcal{U}^{\alpha}_{1},\mathcal{U}^{\alpha}_{2}\in\mathcal{C}(\mathopen{[}0,T\mathclose{]},\mathbb{R}^{m})$ are open (bounded) set.

We notice that, since $\text{\rm deg}_{B}(k^{*},\mathcal{U}\cap \mathbb{R}^{m},0)\neq0$, there exists at least a constant solution in $\mathcal{K}$ and therefore at least one of the sets
$\mathcal{U}^{\alpha}$ contains an element of the form $\omega\in \mathcal{U}\cap\mathbb{R}^{m}$.
This in turn means that at least one of the $\mathcal{U}^{\alpha}_{2}$ contains the element $0$.

Next, we define the set
\begin{equation*}
\Omega := \bigcup_{\alpha=1}^{\ell} \Omega^{\alpha}, \quad \text{where } \, \Omega^{\alpha}:= \mathcal{U}^{\alpha}_{1}\times\phi(\mathcal{U}^{\alpha}_{2}).
\end{equation*}
Clearly the set $\Omega$ is open and bounded in $X$.
Moreover, if $u\in\partial\Omega$ then there exists at least an index $\alpha\in\{1,\ldots,\ell\}$ such that $u\in\partial\Omega^{\alpha}$.
From the above remark, we also have that $0\in\phi(\mathcal{U}^{\alpha}_{2})$ for at least an index $\alpha\in\{1,\ldots,\ell\}$.

Finally, arguing as in \textit{Step~2}, it is easy to check the validity of all the hypotheses of Theorem~\ref{th-3.3} and thus we obtain that
there exists at least a $T$-periodic solution $x(t)=(x_{1}(t),x_{2}(t))$ of \eqref{syst-cycl} in $\overline{\Omega}$
(actually,  $x\in \Omega$).
Defining $u(t):=x_{1}(t)$ for $t\in\mathopen{[}0,T\mathclose{]}$,
we immediately conclude that there is an index $\alpha\in\{1,\ldots,\ell\}$ such that $u=x_{1}\in \mathcal{U}^{\alpha}_{1}$,
$u'=\phi^{-1}(x_{2}) \in \mathcal{U}^{\alpha}_{2}$.
Then $u \in {\mathcal{U}}$, and $u(t)$ satisfies \eqref{eq-phik} and \eqref{2BC}.

The theorem is thus proved.
\end{proof}

\begin{remark}\label{rem-3.2}
In \cite{MaMa-98,MaMa-98T,MaMa-00}, Man\'{a}sevich and Mawhin consider a class of continuous functions $\phi \colon \mathbb{R}^{m} \to \mathbb{R}^{m}$ satisfying
\begin{itemize}
\item[$(H1)$] for every $x_{1},x_{2}\in \mathbb{R}^{m}$, $x_{1}\neq x_{2}$, $\langle\phi(x_{1})-\phi(x_{2}),x_{1}-x_{2}\rangle>0$;
\item[$(H2)$] there exists a function $\alpha \colon \mathopen{[}0,+\infty\mathclose{[} \to \mathopen{[}0,+\infty\mathclose{[}$, with $\alpha(s)\to+\infty$ as $s\to+\infty$,
such that $\langle\phi(x),x\rangle\geq\alpha(|x|)|x|$ for all $x\in\mathbb{R}^{m}$.
\end{itemize}
From these two conditions it follows that the map $\phi \colon \mathbb{R}^{m} \to \mathbb{R}^{m}$ is a homeomorphism such that $\phi(0)=0$.
Clearly our hypotheses cover the case considered in \cite{MaMa-98,MaMa-98T,MaMa-00} and, moreover, it is more general as explained below and in Figure~\ref{fig-01}.

Under our conditions we can deal with a continuous function built in the following manner.
For $i=1,\ldots,n$, let $h_{i}\colon \mathbb{R} \to \mathbb{R}$ be a homeomorphism such that $h_{i}(0)=0$.
Let $\mathcal{M}\in GL_{n}(\mathbb{R})$ be an invertible matrix in $\mathbb{R}^{m\times m}$.
The function $\phi \colon \mathbb{R}^{m} \to \mathbb{R}^{m}$ defined as
\begin{equation*}
\begin{pmatrix} u_{1} \\ \vdots \\ u_{n} \end{pmatrix}
\mapsto
\mathcal{M}
\begin{pmatrix} h_{1}(u_{1}) \\ \vdots \\ h_{n}(u_{n}) \end{pmatrix}
\end{equation*}
is a homeomorphism.
Figure~\ref{fig-01} shows another type of homeomorphism
$\phi \colon \mathbb{R}^{2} \to \mathbb{R}^{2}$ which does not satisfy conditions $(H1)$ and $(H2)$.
$\hfill\lhd$

\begin{figure}[h!]
\centering
\includegraphics[width=0.45\textwidth]{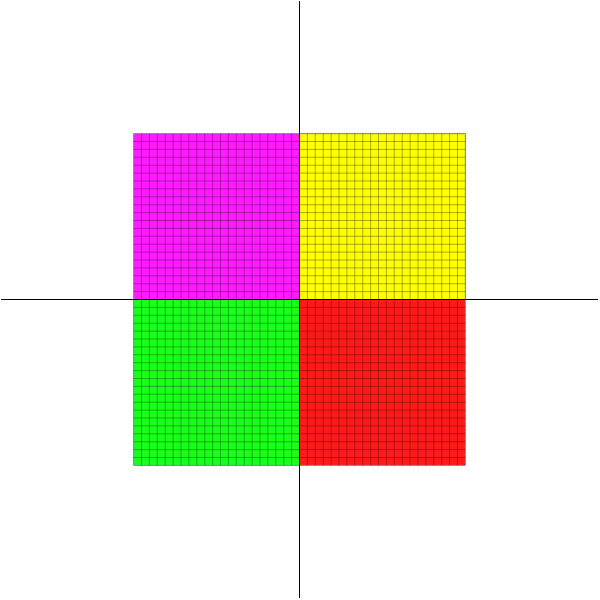}
\quad\quad
\includegraphics[width=0.45\textwidth]{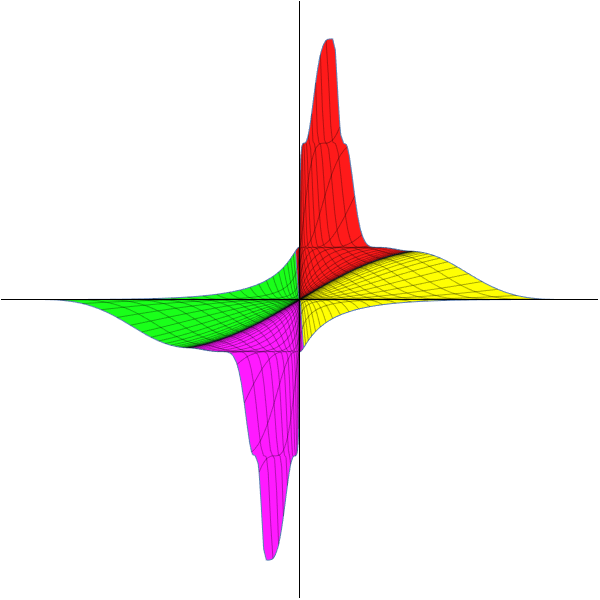}
\caption{\small{The figure shows an example of a homeomorphism $\phi \colon \mathbb{R}^{2} \to \mathbb{R}^{2}$ which does not satisfy conditions $(H1)$ and $(H2)$.
We consider a function $\phi=(\phi_{1},\phi_{2})$ defined as $\phi_{1}(x,y):= (3x+y^{3})^{3}/40$, $\phi_{2}(x,y):= (\sin(2(x-y^{5})^{3}) + 2(x-y^{5})^{3})/10$.
We illustrate how the square $\mathopen{[}-1,1\mathclose{]}\times\mathopen{[}-1,1\mathclose{]} \subseteq \mathbb{R}^{2}$ (on the left) is mapped
by the function $\phi$ (figure on the right).
For example the points $(-1,2)$, $(5,6)$ do not satisfy $(H1)$ and the point $(-2,-2)$ does not satisfy $(H2)$.
}}
\label{fig-01}
\end{figure}
\end{remark}

\section{Periodic solutions to cyclic feedback systems: homotopy to an autonomous system}\label{section-4}

In this section we continue the study of the differential system $(\mathscr{C})$ introduced in Section~\ref{section-3}.
We keep all the basic assumptions for $(\mathscr{C})$ considered therein, as well as the abstract framework for the coincidence degree.
As an application, we give a continuation theorem which involves an homotopy between $(\mathscr{C})$ and
the \textit{autonomous} differential system
\begin{equation*}
\begin{cases}
\, x_{1}' = g_{1}(x_{2}) \\
\, x_{2}' = g_{2}(x_{3}) \\
\, \quad \vdots \\
\, x_{n-1}' = g_{n-1}(x_{n}) \\
\, x_{n}' = h_{0}(x_{1},\ldots,x_{n}),
\end{cases}
\leqno{(\mathscr{C}_{0}^{0})}
\end{equation*}
where $h_{0}\colon \mathbb{R}^{m}\times \dots \times \mathbb{R}^{m} \to \mathbb{R}^{m}$ is a continuous function.
In detail, we consider an auxiliary function
$\tilde{h} \colon \mathopen{[}0,T\mathclose{]}\times\mathbb{R}^{m}\times \dots \times \mathbb{R}^{m}\times \mathopen{[}0,1\mathclose{]} \to \mathbb{R}^{m}$,
satisfying the $L^{1}$-Carath\'{e}odory conditions and such that
\begin{equation}\label{eq-contaut}
\begin{aligned}
&\tilde{h}(t,x_{1},\ldots,x_{n},1) = h(t,x_{1},\ldots,x_{n}), \\
&\tilde{h}(t,x_{1},\ldots,x_{n},0) = h_{0}(x_{1},\ldots,x_{n}).
\end{aligned}
\end{equation}
Consistently with the previous notation, since $M$ is the Nemytskii operator associated with $h$,
we denote by $\tilde{M}$ the corresponding operator associated with $\tilde{h}$.
We also introduce the autonomous vector field
\begin{equation*}
\hat{g}_{0}(s):=\bigl{(}g_{1}(s_{2}),\ldots,g_{n-1}(s_{n}),h_{0}(s)\bigr{)}, \quad s=(s_{1},\ldots,s_{n})\in\mathbb{R}^{mn}.
\end{equation*}

In the next results, when we write  $\Omega \subseteq \text{\rm dom}\,\tilde{M}$ (or $\overline{\Omega} \subseteq \text{\rm dom}\,\tilde{M}$),
we in fact consider only the case of $\Omega \subseteq X =\mathcal{C}(\mathopen{[}0,T\mathclose{]},\mathbb{R}^{mn})$. However,
in principle, the same results could be applied (using Theorem~\ref{th-CMZ-BM})
also to a more general situation, as explained in Remark~\ref{rem-3.1} (see also Section~\ref{section-6}).

We are now in position to state our first result, which is a direct consequence of \cite[Theorem~1]{CaMaZa-92} and of the
homotopic invariance of the coincidence degree (cf.~Lemma~\ref{lemma-inv}).

\begin{lemma}\label{lem-cycl0}
Let $\Omega$ be an open and bounded set with $\overline{\Omega} \subseteq \text{\rm dom}\,\tilde{M}$.
Suppose that the following condition holds.
\begin{itemize}
\item[$(a'_{1})$]
For each $\lambda\in\mathopen{[}0,1\mathclose{]}$ there is no $T$-periodic solution of
\begin{equation*}
\begin{cases}
\, x_{1}' = g_{1}(x_{2}) \\
\, x_{2}' = g_{2}(x_{3}) \\
\, \quad \vdots \\
\, x_{n-1}' = g_{n-1}(x_{n}) \\
\, x_{n}' = \tilde{h}(t,x_{1},\ldots,x_{n},\lambda)
\end{cases}
\leqno{(\mathscr{C}^{0}_{\lambda})}
\end{equation*}
with $x\in \partial\Omega$.
\end{itemize}
Then
\begin{equation*}
D_{L}(L-M,\Omega) = (-1)^{mn} \, \text{\rm deg}_{B}(\hat{g}_{0},\Omega \cap \mathbb{R}^{mn},0).
\end{equation*}
\end{lemma}

From this result the next one follows immediately (see \cite[Theorem~2]{CaMaZa-92} and Theorem~\ref{th-1.2} in the introduction).

\begin{theorem}\label{th-4.1}
Let $\Omega$ be an open and bounded set with $\overline{\Omega} \subseteq \text{\rm dom}\,\tilde{M}$.
Suppose that the following conditions hold.
\begin{itemize}
\item For each $\lambda\in\mathopen{[}0,1\mathclose{[}$ there is no $T$-periodic solution of
$(\mathscr{C}^{0}_{\lambda})$
with $x\in \partial\Omega$.
\item $\text{\rm deg}_{B}(\hat{g}_{0},\Omega \cap \mathbb{R}^{mn},0)\neq0$.
\end{itemize}
Then there exists at least a $T$-periodic solution of $(\mathscr{C})$ in $\overline{\Omega}$.
\end{theorem}

In this manner, we reduce part of our problem to the study of the degree of $\hat{g}_{0}$ and for this purpose
we can take advantage of the conditions considered in the previous section. We can thus produce results
analogous to Corollary~\ref{cor-3.3} and Corollary~\ref{cor-3.4}.
With this respect, it is convenient to introduce the function
$h^{*}_{0}\colon\mathbb{R}^{m}\to\mathbb{R}^{m}$ defined as
\begin{equation*}
h^{*}_{0}(\omega):= h_{0}(\omega,0,\ldots,0), \quad \omega\in \mathbb{R}^{m}.
\end{equation*}

The analogous of Corollary~\ref{cor-3.4} is the following, where the open sets $\mathcal{O}_{i}$ are defined as in the previous section.

\begin{corollary}\label{cor-4.1}
Let $\Omega$ be an open and bounded set with $\overline{\Omega} \subseteq \text{\rm dom}\,\tilde{M}$.
Assume $(a'_{1})$ and also the following conditions
\begin{itemize}
\item [$(a'_{6})$] $(h^{*}_{0})^{-1}(0)\cap \partial\mathcal{O}_{1} = \emptyset$;
\item [$(a'_{*})$] $0 \in \mathcal{O}_{i+1}$ and  $g_{i}|_{\overline{\mathcal{O}_{i+1}}}
\colon \overline{\mathcal{O}_{i+1}}\to g_{i}(\overline{\mathcal{O}_{i+1}})\subseteq \mathbb{R}^{m}$
is a homeomorphism with $g_{i}(0) =0$, for all $i=1,\ldots,n-1$.
\end{itemize}
If
\begin{equation*}
\text{\rm deg}_{B}(h_{0}^{*},\mathcal{O}_{1},0)\neq0,
\end{equation*}
then there exists at least a $T$-periodic solution of $(\mathscr{C})$ in $\overline{\Omega}$.
\end{corollary}

\medskip

All the results presented in this section can be stated also in the case of
an open possibly unbounded set $\Omega \subseteq \text{\rm dom}\,\tilde{M}$.
For this aim, one have to follow the scheme presented in Section~\ref{section-3}
and a modification of \cite[Theorem~1]{BaMa-91} for open (not necessarily bounded) sets,
which is discussed in Appendix~\ref{appendix-A}.

In particular, the analogous of Theorem~\ref{th-4.1} is the following result (corresponding to Theorem~\ref{th-cycl1}).

\begin{theorem}\label{th-4.2}
Let $\Omega \subseteq \text{\rm dom}\,M$ be an open (possibly unbounded) set.
Suppose that the following conditions hold.
\begin{itemize}
\item There exists a compact set $\mathcal{K}\subseteq \Omega$ containing all the possible $T$-periodic solutions of $(\mathscr{C}^{0}_{\lambda})$ for any $\lambda\in\mathopen{[}0,1\mathclose{]}$.
\item $\text{\rm deg}_{B}(\hat{g}_{0},\Omega \cap \mathbb{R}^{mn},0)\neq0$.
\end{itemize}
Then there exists at least a $T$-periodic solution of $(\mathscr{C})$ in $\Omega$.
\end{theorem}

\medskip

Now, at this point, we can repeat (almost step by step) the results obtained in Section~\ref{section-3}
for system \eqref{eq-3.1}. The only difference is that the continuation theorem will make use of an homotopy
leading system \eqref{syst-cycl} to an autonomous system of the form
\begin{equation*}
\begin{cases}
\, x_{1}' = \varphi_{1}^{-1}(x_{2}) \\
\, x_{2}' = \varphi_{2}^{-1}(x_{3}) \\
\, \quad \vdots \\
\, x_{n-1}' = \varphi_{n-1}^{-1}(x_{n}) \\
\, x_{n}' = h_{0}(x_{1},\ldots,x_{n}).
\end{cases}
\end{equation*}
In this setting, the analogous of Theorem~\ref{th-3.3} is the next result, where the function $\tilde{h}$ is defined as in \eqref{eq-contaut}.

\begin{theorem}\label{th-4.3}
Let $\Omega \subseteq \mathcal{C}(\mathopen{[}0,T\mathclose{]},\mathbb{R}^{mn})$ be an open and bounded set
such that $0\in \mathcal{O}_{i}$ for all $i=2,\ldots,n$.
Suppose that
\begin{itemize}
\item for each $\lambda\in\mathopen{[}0,1\mathclose{[}$ there is no
$T$-periodic solutions of
\begin{equation*}
\begin{cases}
\, x_{1}' = \varphi_{1}^{-1}(x_{2}) \\
\, x_{2}' = \varphi_{2}^{-1}(x_{3}) \\
\, \quad \vdots \\
\, x_{n-1}' = \varphi_{n-1}^{-1}(x_{n}) \\
\, x_{n}' = \tilde{h}(t,x_{1},\ldots,x_{n},\lambda)
\end{cases}
\end{equation*}
with $x\in \partial\Omega$;
\item $h^{*}_{0}(\omega) \neq 0$, for every $\omega\in \partial\mathcal{O}_{1}$ and $\text{\rm deg}_{B}(h^{*}_{0},\mathcal{O}_{1},0)\neq 0$.
\end{itemize}
Then there exists at least a $T$-periodic solution $x(t)$ of \eqref{syst-cycl} in $\overline{\Omega}$.
\end{theorem}

The proof is omitted as it is a direct consequence of Corollary~\ref{cor-4.1}.

\medskip

As in the final part of Section~\ref{section-3}, we propose an application to the second order $\phi$-Laplacian equation
\eqref{eq-phik}. In the present case, instead of equation \eqref{eq-phiklam}, we consider the system
\begin{equation}\label{eq-klau}
\bigl{(}\phi(u')\bigr{)}' + \tilde{k}(t,u,u',\lambda) = 0,
\end{equation}
where $\tilde{k} \colon \mathopen{[}0,T\mathclose{]} \times \mathbb{R}^{m} \times \mathbb{R}^{m} \times \mathopen{[}0,1\mathclose{]} \to \mathbb{R}^{m}$
is an $L^{1}$-Carath\'{e}odory function such that
\begin{equation*}
\tilde{k}(t,x_{1},x_{2},1) = k(t,x_{1},x_{2}), \qquad \tilde{k}(t,x_{1},x_{2},0) = k_{0}(x_{1},x_{2}),
\end{equation*}
with $k_{0} \colon \mathbb{R}^{m} \times \mathbb{R}^{m} \to \mathbb{R}^{m}$ an autonomous field.

Now we are in a position to prove the following continuation theorem which corresponds to \cite[Theorem~4.1]{MaMa-98}
(but without any additional assumption on the homeomorphism $\phi$).

\begin{theorem}\label{th-mama2}
Let $\mathcal{U}$ be an open and bounded set in $\mathcal{C}^{1}_{T}$ such that the following conditions hold.
\begin{itemize}
\item For each $\lambda\in\mathopen{[}0,1\mathclose{[}$ the problem
\begin{equation*}
\bigl{(}\phi(u')\bigr{)}' + \tilde{k}(t,u,u',\lambda) = 0, \quad u(0) = u(T), \quad u'(0) = u'(T),
\end{equation*}
has no solution on $\partial\mathcal{U}$.
\item The Brouwer degree
\begin{equation*}
\text{\rm deg}_{B}(k_{0}(\cdot,0),\mathcal{U}\cap \mathbb{R}^{m},0)\neq 0.
\end{equation*}
\end{itemize}
Then, problem \eqref{eq-phik}-\eqref{2BC} has at least a solution in $\overline{\mathcal{U}}$.
\end{theorem}

\begin{proof}
If there exists a solution in $\partial\mathcal{U}$, we are done.
Then, for the rest of the proof, we assume that problem \eqref{eq-phik}-\eqref{2BC} has no solution in $\partial\mathcal{U}$.
We split our argument into three steps, which are the same as in the proof of Theorem~\ref{th-mama1}.

\smallskip

\noindent
\textit{Step~1. Compactness. }
The set
\begin{equation*}
\mathcal{K} := \bigcup_{\lambda\in\mathopen{]}0,1\mathclose{]}}\Bigl{\{}u\in\overline{\mathcal{U}}\colon \bigl{(}\phi(u')\bigr{)}' +
\tilde{k}(t,u,u',\lambda) = 0 \Bigr{\}}
\end{equation*}
is a compact subset of $\mathcal{U}$. To check this claim we just repeat (with obvious changes) the proof of
\textit{Step~1} of Theorem~\ref{th-mama1}.

\smallskip

\noindent
\textit{Step~2. A special case for the domain. }
Suppose that there exist two open bounded sets
$\mathcal{U}_{1},\mathcal{U}_{2}\in\mathcal{C}(\mathopen{[}0,T\mathclose{]},\mathbb{R}^{m})$ with $0\in\mathcal{U}_{2}$ such that
\begin{equation*}
\mathcal{U} = \bigl{\{}u\in\mathcal{C}^{1}_{T} \colon u\in\mathcal{U}_{1}, \, u'\in\mathcal{U}_{2} \bigr{\}}.
\end{equation*}
We write \eqref{eq-klau} as an equivalent first order cyclic system in $\mathbb{R}^{2m}$ of the form
\begin{equation*}
\begin{cases}
\, x'_{1} = \phi^{-1}(x_{2}) \\
\, x'_{2} = \tilde{h}(t,x_{1},x_{2},\lambda),
\end{cases}
\end{equation*}
where $\tilde{h}(t,x_{1},x_{2},\lambda):= -\tilde{k}(t,x_{1},\phi^{-1}(x_{2}),\lambda)$.

In the Banach space $X:= \mathcal{C}(\mathopen{[}0,T\mathclose{]},\mathbb{R}^{2m})$ we define the open and bounded set
\begin{equation*}
\Omega := \mathcal{U}_{1} \times \phi(\mathcal{U}_{2}) = \bigl{\{} x=(x_{1},x_{2})
\in X \colon x_{1} \in \mathcal{U}_{1}, \, \phi^{-1}(x_{2}) \in \mathcal{U}_{2} \bigr{\}}.
\end{equation*}
We can now apply Theorem~\ref{th-4.3} (analogously as we applied Theorem~\ref{th-3.3} in \textit{Step~2} in the proof of Theorem~\ref{th-mama1})
and obtain the existence of at least a $T$-periodic solution $(x_{1}(t),x_{2}(t))\in \Omega$ of
\begin{equation*}
\begin{cases}
\, x'_{1} = \phi^{-1}(x_{2}) \\
\, x'_{2} = h(t,x_{1},x_{2}),
\end{cases}
\end{equation*}
for $h(t,x_{1},x_{2}):= -k(t,x_{1},\phi^{-1}(x_{2}))$. Then, the first component $u:=x_{1}$
of such a solution is a solution of \eqref{eq-phik} with $u\in \mathcal{U}$.

\smallskip

\noindent
\textit{Step~3. General case. }
Let $\mathcal{U}\in\mathcal{C}^{1}_{T}$ be an open and bounded set.
Recalling \textit{Step~1}, we know that $\mathcal{K}$ is a compact subset of $\mathcal{U}$ and we cover it by a finite
number of open sets $\mathcal{U}^{\alpha}$ as in \eqref{sets-u}. From this point on, the proof follows the same
lines as those of \textit{Step~3} in the proof of Theorem~\ref{th-mama1} and we can conclude.
\end{proof}

\section{Periodic solutions to Hartman-type perturbations of the $\phi$-Laplacian operator}\label{section-5}

As an application of the previous continuation results, we propose a further variant of the Hartman-Knobloch
theorem for the $T$-periodic problem associated with the vector second order differential equation
\begin{equation}\label{eq-Ma}
\bigl{(}\phi(u')\bigr{)}' + f(t,u) = 0,
\end{equation}
where $f \colon \mathopen{[}0,T\mathclose{]} \times \mathbb{R}^{m} \to \mathbb{R}^{m}$
is a continuous function and $\phi \colon \mathbb{R}^{m} \to \mathbb{R}^{m}$ is a homeomorphism with $\phi(0)=0$.

A classical result of Hartman \cite{Ha-60} guarantees the existence of a solution for the two-point (Dirichlet) boundary value problem
associated with
\begin{equation}\label{eq-5.1}
u'' + f(t,u) = 0,
\end{equation}
by assuming the existence of a constant $R > 0$ such that
\begin{equation}\label{Hcond}
\langle f(t,\xi),\xi\rangle \leq 0, \quad \forall \, t\in \mathopen{[}0,T\mathclose{]},
\; \forall \, \xi\in\mathbb{R}^{m} \text{ with } \|\xi\|_{\mathbb{R}^{m}}= R.
\end{equation}
Under the same conditions,  Knobloch in \cite{Kn-71} obtained an existence result for the periodic problem and with a Lipschitzian $f$
(see also \cite{RoMa-73}, dealing with a continuous $f$). Both Hartman and Knobloch results apply also to
more general systems of the form
\begin{equation}\label{eq-5.2}
u'' + f(t,u,u') = 0,
\end{equation}
under suitable growth conditions on $u'$ of Bernstein-Nagumo type.
For simplicity, we do not pursue our study in this direction and refer to \cite{DCHa-06,Ma-81}
for interesting surveys and information on this topic.

In \cite{Ma-00}, Mawhin extended the theorems for equation \eqref{eq-5.1}
to systems of the form \eqref{eq-Ma} for a $p$-Laplacian differential operator, namely for $\phi(\xi) = \psi_{p}(\xi)$,
where
\begin{equation}\label{eq-psip}
\psi_{p}(\xi):= |\xi|^{p-2}\xi, \; \text{ if } \, \xi\in \mathbb{R}^{m}\setminus\{0\}, \qquad \psi_{p}(0) = 0,
\end{equation}
for $p > 1$ (see also \cite{Ma-01}). The corresponding results for system \eqref{eq-5.2} were generalized to the $p$-Laplacian operator by
Mawhin and Ure\~{n}a in \cite{MaUr-02}.

We plan now to present a version of Knobloch theorem, limited to the case of \eqref{eq-5.1}, for system \eqref{eq-Ma}
and involving a class of nonlinear differential operators which are not included in those studied in \cite{MaMa-98, Ma-01}.
In any case, we shall borrow
some arguments already developed in \cite{Ma-00} and \cite{MaUr-02} in the case of $\phi = \psi_{p}$.
Thus our computations are in debt of those performed in the above quoted papers; indeed
we show that they can be
reproduced in our more general setting, by virtue of the continuation theorems developed in the
previous sections.

\medskip

In the following lemma we obtain an a~priori bound for the derivative of the solution to a parameter-dependent
equation of the form
\begin{equation}\label{eq-lambda}
\bigl{(}\phi(u')\bigr{)}' + \tilde{f}(t,u,\lambda) = 0,
\end{equation}
where
$\tilde{f} \colon \mathopen{[}0,T\mathclose{]} \times \mathbb{R}^{m}
\times \mathopen{[}0,1\mathclose{]} \to \mathbb{R}^{m}$ is a continuous function.

Using the continuity of $\tilde{f}$, for any constant $d > 0$ we define
\begin{equation*}
C_{d} := \max \bigl{\{}\|f(t,\xi,\lambda)\|_{\mathbb{R}^{m}} \colon t\in\mathopen{[}0,T\mathclose{]}, \, \xi\in B[0,d], \, \lambda \in\mathopen{[}0,1\mathclose{]} \bigr{\}}.
\end{equation*}

\begin{lemma}\label{lem-apriori}
Let us suppose that
\begin{equation*}
\lim_{|\xi|\to+\infty} \langle\phi(\xi),\xi\rangle = +\infty.
\end{equation*}
Then, for every $d>0$ there exists $M_{d}>0$ such that
$\|u'\|_{\infty}<M_{d}$, whenever $u(t)$ is a
$T$-periodic solution $u(t)$ of \eqref{eq-lambda}, for some $\lambda\in \mathopen{[}0,1\mathclose{]}$,  with $\|u\|_{\infty}\leq d$.
\end{lemma}

\begin{proof}
Let $u(t)$ be a $T$-periodic solution of \eqref{eq-lambda} with $\|u\|_{\infty}\leq d$.
We divide the proof into two steps.

\smallskip

\noindent
\textit{Step~1. } We claim that there exist a point $t_{0}\in\mathopen{[}0,T\mathclose{]}$ and a constant
$L_{d}>0$ such that $\|u'(t_{0})\|_{\mathbb{R}^{m}}\leq L_{d}$.

By multiplying equation \eqref{eq-lambda}
by $u(t)$ and by integrating in $\mathopen{[}0,T\mathclose{]}$, we obtain
\begin{equation*}
- \int_{0}^{T} \langle\bigl{(}\phi(u'(t))\bigr{)}',u(t)\rangle~\!dt = \int_{0}^{T} \langle\tilde{f}(t,u(t),\lambda),u(t)\rangle~\!dt
\end{equation*}
and thus
\begin{equation*}
\dfrac{1}{T}\int_{0}^{T} \langle\phi(u'(t)),u'(t)\rangle~\!dt
= \dfrac{1}{T} \int_{0}^{T} \langle\tilde{f}(t,u(t),\lambda),u(t)\rangle~\!dt \leq d C_{d}.
\end{equation*}
By the mean value theorem, there exists $t_{0}\in\mathopen{[}0,T\mathclose{]}$ such that
\begin{equation*}
\langle\phi(u'(t_{0})),u'(t_{0})\rangle \leq d C_{d}.
\end{equation*}
From the hypothesis of the lemma, there exists $L_{d}>0$ such that if
 $\|\xi\|_{\mathbb{R}^{n}}>L_{d}$ then $\langle\phi(\xi),\xi\rangle>d C_{d}$.
Therefore we conclude that $\|u'(t_{0})\|_{\mathbb{R}^{n}}\leq L_{d}$.

\smallskip

\noindent
\textit{Step~2. }
By integrating \eqref{eq-lambda} in $\mathopen{[}t_{0},t\mathclose{]}$, for $t\in \mathopen{[}0,T\mathclose{]}$,
we deduce
\begin{equation*}
\phi(u'(t)) = \phi(u'(t_{0})) - \int_{t_{0}}^{t} \tilde{f}(s,u(s),\lambda)~\!ds
\end{equation*}
and thus
\begin{equation*}
\|\phi(u'(t))\|_{\mathbb{R}^{m}} \leq \|\phi(u'(t_{0}))\|_{\mathbb{R}^{m}} + T C_{d}.
\end{equation*}
Next, by defining
\begin{equation*}
K_{d} := \sup\bigl{\{} \|\phi(\xi)\|_{\mathbb{R}^{m}}\colon
\xi\in\mathbb{R}^{m}, \, \|\xi\|_{\mathbb{R}^{m}}\leq L_{d}
\bigr{\}},
\end{equation*}
we have
\begin{equation*}
\|\phi(u'(t))\|_{\mathbb{R}^{m}} \leq K_{d} + T C_{d}, \quad \forall \, t\in \mathopen{[}0,T\mathclose{]},
\end{equation*}
and hence
\begin{equation*}
u'(t) \in \phi^{-1}(B[0,K_{d} + T C_{d}]), \quad \forall \, t\in \mathopen{[}0,T\mathclose{]}.
\end{equation*}
Now, it is sufficient to take as $M_{d}>0$ any real number such that
\begin{equation*}
\phi^{-1}(B[0,K_{d} + T C_{d}]) \subseteq B(0,M_{d}).
\end{equation*}
The lemma is thus proved.
\end{proof}

The following existence result holds.

\begin{theorem}\label{th-5.1}
Let $A\colon\mathbb{R}^{m}\setminus\{0\}\to\mathopen{]}0,+\infty\mathclose{[}$ be a continuous function and let
\begin{equation}\label{eq-Aphi}
\phi(\xi) := A(\xi) \xi, \; \text{ if } \, \xi\in \mathbb{R}^{m}\setminus\{0\}, \qquad \phi(0) = 0.
\end{equation}
Suppose that $\phi \colon \mathbb{R}^{m} \to \mathbb{R}^{m}$ is a homeomorphism.
If there exists $R>0$ such that
Hartman's condition \eqref{Hcond} holds,
then there exists at least a $T$-periodic solution $u(t)$ of \eqref{eq-Ma} such that $\|u(t)\|_{\mathbb{R}^{m}}\leq R$
for all $t\in \mathopen{[}0,T\mathclose{]}$.
\end{theorem}

\begin{proof}
We shall propose two different proofs. The first one is based on Theorem~\ref{th-mama2}
and is partially inspired by the approach introduced by Mawhin in \cite{Ma-00}. The second one will be
only sketched and is based on Theorem~\ref{th-mama1}, following the approach in \cite{MaUr-02}.

We introduce a function $\tilde{f}(t,\xi,\lambda)$ such that
\begin{equation*}
\tilde{f}(t,\xi,1) = f(t,\xi), \quad \tilde{f}(t,\xi,0) = - \xi, \qquad t\in \mathopen{[}0,T\mathclose{]}, \; \xi\in\mathbb{R}^{m},
\end{equation*}
and
\begin{equation}\label{eq-5.3}
\langle\tilde{f}(t,\xi,\lambda),\xi\rangle < 0,
\quad \forall \, t\in \mathopen{[}0,T\mathclose{]},
\; \forall \, \xi\in\mathbb{R}^{m} \text{ with } \|\xi\|_{\mathbb{R}^{m}}= R, \;\forall\, \lambda\in \mathopen{[}0,1\mathclose{[}.
\end{equation}
In view of Hartman's condition \eqref{Hcond}, a suitable choice of $\tilde{f}$ could be
\begin{equation*}
\tilde{f}(t,\xi,\lambda) = \lambda f(t,\xi) - (1-\lambda) \xi.
\end{equation*}

Since $\phi \colon \mathbb{R}^{m} \to \mathbb{R}^{m}$ is a homeomorphism, it follows that $|\phi(\xi)|\to +\infty$
as $|\xi|\to +\infty$. Then, by the structure of $\phi$ we have chosen, we have
that $A(\xi)|\xi| \to +\infty$ as $|\xi|\to +\infty$. Hence, $\langle \phi(\xi),\xi\rangle =
A(\xi)|\xi|^{2} \to +\infty$ as $|\xi|\to +\infty$. Therefore, the hypothesis of Lemma~\ref{lem-apriori}
is satisfied and thus there exists a constant $M_{R} > 0$ such that
$\|u'\|_{\infty} < M_{R}$ for any $T$-periodic solution $u(t)$ of \eqref{eq-lambda}
(for some $\lambda \in \mathopen{[}0,1\mathclose{]}$) such that $\|u\|_{\infty} \leq R$.

We are going to apply Theorem~\ref{th-mama2} to the set
\begin{equation}\label{set-uhk}
\mathcal{U} := \bigl{\{}u\in\mathcal{C}^{1}_{T} \colon \|u\|_{\infty} < R, \, \|u'\|_{\infty} < M_{R} \bigr{\}}.
\end{equation}

First, we prove that the $T$-periodic problem associated with \eqref{eq-lambda}, for $\lambda\in \mathopen{[}0,1\mathclose{[}$,
has no solution on $\partial\mathcal{U}$.
As already observed, from Lemma~\ref{lem-apriori} we deduce that every solution $u\in\overline{\mathcal{U}}$ satisfies $\|u'\|_{\infty} < M_{R}$.
Therefore, to prove our claim we have only to verify that if $u\in\overline{\mathcal{U}}$ is a $T$-periodic solution of \eqref{eq-lambda}
for some $\lambda \in \mathopen{[}0,1\mathclose{[}$, then $\|u\|_{\infty} < R$.

By contradiction, assume that there exists a $T$-periodic solution $u(t)$ of \eqref{eq-lambda} (for  some $\lambda \in \mathopen{[}0,1\mathclose{[}$)
such that $\|u\|_{\infty} = R$.
Then, there exists $t^{*}\in\mathopen{[}0,T\mathclose{]}$ such that $\|u(t^{*})\|_{\mathbb{R}^{m}} = R$,
with $t^{*}$ a point of maximum for $\|u(t)\|_{\mathbb{R}^{m}}$. By the $T$-periodicity of the map $t\mapsto \|u(t)\|^{2}_{\mathbb{R}^{m}}$,
we have that $t^{*}$ is a critical point and therefore (by differentiating) $\langle u'(t^{*}),u(t^{*})\rangle = 0$.
Next, we observe that, if $u'(t^*) = 0$, then $\phi(u'(t^{*})) = 0$, while, if $u'(t^{*})\neq 0$, then
(by the particular form of $\phi$)
\begin{equation*}
\langle\phi(u'(t^{*})),u(t^{*})\rangle = A(u'(t^{*})) \langle u'(t^{*}),u(t^{*})\rangle=0.
\end{equation*}
From the equality
\begin{equation*}
\dfrac{d}{dt} \langle\phi(u'(t)),u(t)\rangle =
-\langle\tilde{f}(t,u(t),\lambda),u(t)\rangle + \langle\phi(u'(t)),u'(t)\rangle
\end{equation*}
and condition \eqref{eq-5.3}, we obtain that
\begin{equation*}
\dfrac{d}{dt} \langle\phi(u'(t)),u(t)\rangle |_{t=t^{*}} > 0.
\end{equation*}
We thus have proved that the function $v(t):= \langle\phi(u'(t)),u(t)\rangle$ is such that
$v(t^{*}) = 0$ and $v'(t^{*}) > 0$.
We deduce the existence of $\varepsilon>0$ such that
\begin{align*}
&v(t) < 0, \quad \text{for all } t\in\mathopen{]}t^{*}-\varepsilon,t^{*}\mathclose{[},
\\ & v(t) >0, \quad \text{for all } t\in\mathopen{]}t^{*},t^{*}+\varepsilon\mathclose{[}.
\end{align*}
Both the above inequalities are meaningful also
if $t^{*} = 0$ or if $t^{*} = T$, because, in this case, $|u(0)| = |u(T)| = \|u\|_{\infty} = R$
and also $v(0) = v(T)$. More precisely, if such a situation occurs, we read the first inequality
for $t^{*} = T$ and the second one for $t^{*} = 0$.
The special form of $\phi$ implies that
\begin{align*}
&\langle u'(t),u(t)\rangle < 0, \quad \text{for all } t\in\mathopen{]}t^{*}-\varepsilon,t^{*}\mathclose{[},
\\ & \langle u'(t),u(t)\rangle >0, \quad \text{for all } t\in\mathopen{]}t^{*},t^{*}+\varepsilon\mathclose{[}.
\end{align*}
Then, since
\begin{equation*}
\dfrac{d}{dt}\|u(t)\|_{\mathbb{R}^{m}}^{2} = 2 \, \langle u'(t),u(t) \rangle,
\end{equation*}
we obtain that $t=t^{*}$ cannot be a maximum point for the function $\|\cdot\|_{\mathbb{R}^{m}}$, a contradiction.

The second hypothesis of Theorem~\ref{th-mama2} follows directly from the fact that
\begin{equation*}
\tilde{f}(t,\xi,0) = -\xi
\end{equation*}
and, clearly, the degree $\text{\rm deg}_{B}(-Id_{\mathbb{R}^{m}},\mathcal{U}\cap \mathbb{R}^{m},0)$ is nonzero.

The theorem is thus proved as an application of Theorem~\ref{th-mama2}.

\smallskip

The same theorem can be proved also using Theorem~\ref{th-mama1}. We give just a sketch of the proof.
As in \cite{MaUr-02} we suppose that Hartman's condition
\eqref{Hcond} holds with a strict inequality, namely
\begin{equation}\label{cond-hmu}
\langle f(t,\xi),\xi\rangle < 0, \quad \forall \, t\in \mathopen{[}0,T\mathclose{]},
\; \forall \, \xi\in\mathbb{R}^{m} \text{ with } \|\xi\|_{\mathbb{R}^{m}}= R.
\end{equation}
Then, using the same argument as above, we prove that system \eqref{eq-lambda} for
\begin{equation*}
\tilde{f}(t,\xi,\lambda):= \lambda f(t,\xi)
\end{equation*}
satisfies the following condition: for each $\lambda \in \mathopen{]}0,1\mathclose{[}$ there are no ($T$-periodic) solutions on the boundary of $\mathcal{U}$
(where $\mathcal{U}$ is defined as in \eqref{set-uhk}).
Moreover, consistently with the notation in Theorem~\ref{th-mama1}, we have
\begin{equation*}
k^{*}(\omega):= \frac{1}{T}\int_{0}^{T} f(t,\omega)~\!dt, \quad \omega\in\mathbb{R}^{m},
\end{equation*}
and, by \eqref{cond-hmu}, we obtain $\langle k^{*}(\omega),\omega\rangle < 0$ for all $\omega\in \partial B(0,R) = \partial\mathcal{U}
\cap \mathbb{R}^{m}$. Hence,
\begin{equation*}
\text{\rm deg}_{B}(k^{*},\mathcal{U}\cap \mathbb{R}^{m},0)= \text{\rm deg}_{B}(-Id_{\mathbb{R}^{m}},\mathcal{U}\cap \mathbb{R}^{m},0)
= (-1)^{m}\not= 0.
\end{equation*}
At this point Theorem~\ref{th-mama1} implies the existence of at least a $T$-periodic solution of \eqref{eq-Ma} with $\|u\|_{\infty} \leq R$
(and also $\|u'\|_{\infty} \leq M_{R}$).

Since the result is obtained under the strict inequality \eqref{cond-hmu} in Hartman's condition, it remains to prove the
theorem within the original inequality \eqref{Hcond}. To achieve this latter step, we approximate the vector field with
functions of the form $f(t,\xi) - \varepsilon \xi$ (with $\varepsilon \to 0^{+}$) and use the a priori bounds for the solutions.
We skip this part since it has been already fully developed in \cite{MaUr-02}.
\end{proof}

\begin{remark}\label{rem-5.1}
Obviously any vector $p$-Laplacian differential operator defined through a homeomorphism $\psi_{p}$,
defined as in \eqref{eq-psip} for $p > 1$, satisfies the assumption of $\phi$ in Theorem~\ref{th-5.1}.
On the other hand, it is possible to provide simple examples of homeomorphisms satisfying \eqref{eq-Aphi}
which do not belong to the class of the $\psi_{p}$-functions considered in \eqref{eq-psip}.
For instance, the map
\begin{equation}\label{arctan}
\phi(\xi) := (\arctan |\xi|)\xi, \quad \xi\in\mathbb{R}^{m},
\end{equation}
fits well for Theorem~\ref{th-5.1} and is not in the $p$-Laplacian class.

The homeomorphism defined in \eqref{arctan}
is a special case of a class of maps of the form
\begin{equation}\label{eq-gamma}
\phi(\xi) : = \gamma(|\xi|)\xi, \; \text{ if } \, \xi\in \mathbb{R}^{m}\setminus\{0\}, \qquad \phi(0) = 0,
\end{equation}
with $\gamma(s)$ a positive continuous function defined for $s > 0$. Such class of operators is clearly included in that
of the form \eqref{eq-Aphi} and it has been considered in \cite{LQC-11} for the singular case, namely for $\phi$ defined on an open ball
$B(0,a)$ and, consequently, for $\gamma(s)$ with
$0 < s < a < +\infty$.

\smallskip

A natural question which raises in this context is whether the homeomorphisms $\phi$
of the form \eqref{eq-Aphi} (and thus, in particular, \eqref{eq-gamma}) belong to the class of nonlinear operators introduced by Man\'{a}sevich and Mawhin in \cite{MaMa-98}
and satisfying conditions $(H1)$ and $(H2)$ recalled in Remark~\ref{rem-3.2}. With this respect, we observe that $(H2)$ is always satisfied,
since $\langle \phi(x),x\rangle = |\phi(x)|\,|x| \geq \alpha(|x|) |x|$ for the map $\alpha(s)$ defined as
\begin{equation*}
\alpha(s):= \min\bigl{\{}|\phi(x)| \colon |x| = s \bigr{\}}, \quad s \geq 0,
\end{equation*}
and such that $\alpha(s)\to +\infty$ as $s\to +\infty$.

Concerning condition $(H1)$ we claim that
the homeomorphisms defined by \eqref{eq-gamma} satisfy this condition.
Indeed, for the proof of $(H1)$ we proceed as follows.
First of all, we notice that a homeomorphism of the form \eqref{eq-gamma} transforms any radial line
$\{\vartheta\vec{v}\colon \vartheta \geq 0\}$ homeomorphically onto itself
(where $\vec{v}\in S^{m-1}$ is an arbitrary unit vector). Hence (by the positivity of $\gamma$) we immediately deduce that
the map $\zeta\colon{\mathbb R}^{+} \to {\mathbb R}^{+}$ defined by $\zeta(s):= \gamma(s)s$ for $s > 0$ and $\zeta(0) = 0$
is an increasing homeomorphism of ${\mathbb R}^{+}$ onto itself. Then, to conclude, we just repeat (with minor modifications)
the proof already given
by Man\'{a}sevich and Mawhin in \cite[Example~2.2]{MaMa-98} for the vector $p$-Laplacian, that is
\begin{align*}
& \langle\phi(x_{1})-\phi(x_{2}),x_{1}-x_{2}\rangle = \\
&= \gamma(|x_{1}|)|x_{1}|^{2} + \gamma(|x_{2}|)|x_{2}|^{2} -  \gamma(|x_{1}|)\langle x_{1},x_{2}\rangle - \gamma(|x_{2}|)\langle x_{1},x_{2}\rangle\\
&\geq \gamma(|x_{1}|)|x_{1}|^{2} + \gamma(|x_{2}|)|x_{2}|^{2} -  \gamma(|x_{1}|)|x_{1}| |x_{2}|- \gamma(|x_{2}|)|x_{1}| |x_{2}|\\
&= \bigl(\zeta( |x_{1}| ) - \zeta( |x_{2}| ) \bigr)\bigl(|x_{1}|-|x_{2}|\bigr) \geq 0.
\end{align*}
Therefore, we have that $\langle\phi(x_{1})-\phi(x_{2}),x_{1}-x_{2}\rangle = 0$ implies $|x_{1}| = |x_{2}|$. Then, either
$x_{1} = x_{2} = 0$ or $|x_{1}| \neq 0$ and
$\langle\phi(x_{1})-\phi(x_{2}),x_{1}-x_{2}\rangle = \gamma(|x_{1}|) |x_{1} - x_{2}|^{2}$, so that we conclude
again that $x_{1} = x_{2}$. Hence $(H1)$ is proved.

\smallskip

On the other hand, we claim that there are homomorphisms $\phi$
of the form \eqref{eq-Aphi} which do not satisfy condition $(H1)$.
Before presenting our example, we first introduce a class of homeomorphisms satisfying \eqref{eq-Aphi}
which strictly includes the $p$-Laplacian class.

Let  $\mathcal{A}\colon S^{m-1} \to \mathbb{R}^{+}_{0}$ be a continuous map, where $S^{m-1} = \partial B(0,1)$ is the
$(m-1)$-dimensional sphere in the Euclidean space $\mathbb{R}^{m}$. Let also $p > 1$ be a fixed real number.
We define
\begin{equation}\label{eq-AAphi}
\phi(\xi) := |\xi|^{p-2} \mathcal{A}\biggl{(}\frac{\xi}{|\xi|}\biggr{)}\xi, \; \text{ if } \, \xi\in \mathbb{R}^{m}\setminus\{0\}, \qquad \phi(0) = 0,
\end{equation}
which is of the form \eqref{eq-Aphi} for $A(\xi) := |\xi|^{p-2} \mathcal{A}(\xi/|\xi|)$.
It is straightforward to check that maps of the form \eqref{eq-AAphi} are one-to-one, surjective and continuous,
therefore, by Brouwer invariance of domain theorem, are homeomorphisms of $\mathbb{R}^{m}$.

We give now an example of a planar homeomorphism of the form \eqref{eq-AAphi} which does not satisfy condition $(H1)$.
It is not difficult to extend the example to any dimension $m\geq 2$.
For simplicity, we restrict to the case $p=2$.
First of all, we observe that, given two nontrivial vectors $x_{1} \neq x_{2}$ and for $v_{i}:= x_{i}/|x_{i}|$ ($i=1,2$), we have
\begin{align*}
& \langle\phi(x_{1})-\phi(x_{2}),x_{1}-x_{2}\rangle = \\
&= \langle \mathcal{A}(v_{1})x_{1} - \mathcal{A}(v_{2})x_{2},x_{1}-x_{2}\rangle \\
&= \mathcal{A}(v_{1})|x_{1}|^{2} + \mathcal{A}(v_{2})|x_{2}|^{2} - (\mathcal{A}(v_{1})+\mathcal{A}(v_{2}))\langle x_{1},x_{2}\rangle \\
&= \mathcal{A}(v_{1})|x_{1}|^{2} + \mathcal{A}(v_{2})|x_{2}|^{2} - (\mathcal{A}(v_{1})+\mathcal{A}(v_{2}))|x_{1}||x_{2}|\cos\beta.
\end{align*}
The inner product is clearly positive when $\langle x_{1},x_{2}\rangle = 0$. We show now how to find vectors where it can get
negative values. We take as $\mathcal{A}\colon S^{1} \to \mathbb{R}^{+}_{0}$ any continuous map such that $\mathcal{A}(1) = 1$
and $\mathcal{A}(e^{i\pi/4}) = 6$. Then, for the vectors
\begin{equation*}
x_{1}=(1,0), \quad x_{2}=\rho \, (\sqrt{2}/2,\sqrt{2}/2), \; \text{ with } \rho=\dfrac{7\sqrt{2}}{24},
\end{equation*}
the above formula (with $\beta = \pi/4$) yields
\begin{equation*}
\langle\phi(x_{1})-\phi(x_{2}),x_{1}-x_{2}\rangle = -\dfrac{1}{48} < 0.
\end{equation*}
This shows the effectiveness of Theorem~\ref{th-mama1} and Theorem~\ref{th-mama2} which allow to extend Hartman-Knobloch theorem
to a broader class of operators not included in \cite{MaMa-98, Ma-01}.
$\hfill\lhd$
\end{remark}

\section{Concluding remarks}\label{section-6}

In this paper all the applications of the abstract results to differential systems have been considered in the
context of vector fields which are globally defined on the Euclidean space $\mathbb{R}^{m}$
(cf.~system $(\mathscr{C})$) or have an inverse that is globally defined (cf.~\eqref{syst-cycl}). However, there are some
significant cases of maps which have as their natural domain/image an open (and possibly bounded)
subset of $\mathbb{R}^{m}$. In the one-dimensional case, typical examples in this direction arise from the study of
the mean curvature operator
\begin{equation*}
u \mapsto \dfrac{u'}{\sqrt{1 + (u')^{2}}}
\end{equation*}
or of the Minkowski operator
\begin{equation*}
u \mapsto \dfrac{u'}{\sqrt{1 - (u')^{2}}},
\end{equation*}
which may be described by means of homeomorphisms $\phi \colon I\to J$, with $I= \mathbb{R}$ and $J = \mathopen{]}-1,1\mathclose{[}$, or
$I =\mathopen{]}-1,1\mathclose{[}$ and $J= \mathbb{R}$, respectively
(see, for instance, \cite{BeJeMa-10, ObOm-11} and the references therein).
Higher dimensional versions of these operators are usually studied as well.

As already observed in Remark~\ref{rem-3.1} and underlined many times along the paper,
our abstract setting has been devised in order to deal with these general cases, too.
Indeed, all our results in the previous sections could be reformulated (by suitably adapting the
hypotheses) in terms of operators which are not defined on the whole space. Instead of
discussing again with the needed details all the theorems and lemmas presented above,
we just illustrate how to deal with the autonomous system
\begin{equation}\label{eq-g0}
\bigl{(}\phi(u')\bigr{)}' + k(t,u,u') = 0,
\end{equation}
when $\phi$ (or $\phi^{-1}$) and $k$ are not defined on the whole space.

Let $\mathcal{A},\mathcal{B},\mathcal{G}\subseteq\mathbb{R}^{m}$ be nonempty open connected sets
with $\mathcal{A},\mathcal{B}$ containing the zero element $0_{\mathbb{R}^{m}}$.
Let $k\colon \mathopen{[}0,T\mathclose{]}\times \mathcal{G} \times \mathcal{A} \to \mathbb{R}^{m}$ be an $L^{1}$-Carath\'{e}odory
function and let
$\phi \colon \mathcal{A} \to \phi(\mathcal{A})= \mathcal{B}$ be a homeomorphism with $\phi(0)=0$.

As before, we deal with the $T$-periodic problem associated with equation \eqref{eq-g0}.
We recall that a \textit{$T$-periodic solution of \eqref{eq-g0}} is a continuously differentiable function $u\colon \mathopen{[}0,T\mathclose{]}\to \mathbb{R}^{m}$ such that
$u(0) = u(T)$ and satisfying
\begin{itemize}
\item $u(t)\in\mathcal{G}$, for all $t\in \mathopen{[}0,T\mathclose{]}$;
\item $u'(t)\in\mathcal{A}$, for all $t\in \mathopen{[}0,T\mathclose{]}$;
\item $t\mapsto\phi(u'(t))$ is absolutely continuous;
\item $u(t)$ satisfies \eqref{eq-g0}, for a.e.~$t\in \mathopen{[}0,T\mathclose{]}$.
\end{itemize}

Writing equation \eqref{eq-g0} as an equivalent cyclic feedback system in $\mathbb{R}^{2m}$
in the class $(\mathscr{C})$
\begin{equation}\label{syst-g0}
\begin{cases}
\, u' = \phi^{-1}(y) \\
\, y' = -k(t,u,\phi^{-1}(y)),
\end{cases}
\end{equation}
we can enter the setting presented in Section~\ref{section-2} with the choice of $X$, $Z$ and $L$
described at the beginning of Section~\ref{section-3} (with $n=2$) and with $M$ the Nemytskii
operator associated with the right-hand side of system \eqref{syst-g0}. In this case, $M$
has as a  domain the set of continuous pairs of functions $(x_{1},x_{2})\in X$ such that
$x_{1}(t)\in \mathcal{G}$ and $x_{2}(t)\in \mathcal{B}$ for all $t\in \mathopen{[}0,T\mathclose{]}$.

We notice that our semi-abstract results such as Theorem~\ref{th-cycl2} or Theorem~\ref{th-4.1} can be restated without changes
(in view of the hypothesis $\overline{\Omega} \subseteq \text{\rm dom}\,M$, or  $\overline{\Omega} \subseteq \text{\rm dom}\,\tilde{M}$,
respectively). As a consequence, we can provide versions of Theorem~\ref{th-mama1} and Theorem~\ref{th-mama2} in this more
general context. For these versions we must be careful in the choice of the open and bounded set $\mathcal{U}\subseteq \mathcal{C}^{1}_{T}$
considered in these theorems. Indeed, we need to check that
\begin{equation*}
\overline{\mathcal{U}} \subseteq \bigl{\{}u\in\mathcal{C}^{1}_{T}
\colon u(t)\in\mathcal{G}, \; u'(t)\in\mathcal{A}, \; \forall \, t\in \mathbb{R} \bigr{\}}.
\end{equation*}
In this framework we could also extend to higher dimensions the continuation theorem of
Benevieri, do {\'O} and de Medeiros obtained in \cite{BOM-07} in the one-dimensional case.

\appendix
\section{Coincidence degree results for autonomous equations}\label{appendix-A}

In this appendix we deal with the $T$-periodic boundary value problem associated with the \textit{autonomous equation}
\begin{equation}\label{eq-A.1}
x' = f_{0}(x),
\end{equation}
where $f_{0}\colon \mathcal{A}_{0}\to\mathbb{R}^{d}$ is a continuous function defined on the open (and not necessarily bounded)
set $\mathcal{A}_{0}\subseteq\mathbb{R}^{d}$.

Let $X := \mathcal{C}(\mathopen{[}0,T\mathclose{]},\mathbb{R}^{d})$ be the Banach space of the continuous functions from $\mathopen{[}0,T\mathclose{]}$ to $\mathbb{R}^{d}$,
endowed with the $\sup$-norm, and let $Z:=L^{1}(\mathopen{[}0,T\mathclose{]},\mathbb{R}^{d})$ be the space of integrable functions from $\mathopen{[}0,T\mathclose{]}$ to $\mathbb{R}^{d}$,
endowed with the $L^{1}$-norm. Let $\text{\rm dom}\,L\subseteq X$ be the set of absolutely continuous functions satisfying the periodic boundary condition $x(0)=x(T)$.
We define the linear differential operator $L\colon \text{\rm dom}\,L \to Z$ as
\begin{equation*}
(L x)(t):= x'(t), \quad t\in\mathopen{[}0,T\mathclose{]},
\end{equation*}
which is a Fredholm mapping of index zero.

Let
\begin{equation*}
\text{\rm dom}\,M_{0} := \bigl{\{} x\in X \colon x(t)\in\mathcal{A}_{0}, \, \forall \, t\in \mathopen{[}0,T\mathclose{]}\bigr{\}}.
\end{equation*}
We define $M_{0}\colon \text{\rm dom}\,M_{0}\to Z$ as the Nemytskii operator induced by the function $f_{0}(s)$, namely
\begin{equation*}
(M_{0} u)(t)=f_{0}(u(t)), \quad t\in\mathopen{[}0,T\mathclose{]}.
\end{equation*}

With this position, the $T$-periodic boundary value problem associated with \eqref{eq-A.1} can be written as an equivalent \textit{coincidence equation}
\begin{equation}\label{coinc-auto}
L x = M_{0} x, \quad x\in \text{\rm dom}\,L\cap\text{\rm dom}\,M_{0}.
\end{equation}

\medskip

In this context the following theorem holds.

\begin{theorem}\label{th-CMZ}
Let $\mathcal{A}_{0}=\mathbb{R}^{d}$.
Let $\Omega \subseteq X$ be an open bounded set.
Assume that there is no $u\in\partial\Omega$ such that $u'=f_{0}(u)$.
Then
\begin{equation*}
D_{L}(L-M_{0},\Omega) = (-1)^{d} \, \text{\rm deg}_{B}(f_{0},\Omega\cap\mathbb{R}^{d},0).
\end{equation*}
\end{theorem}

The above result was proved by Capietto, Mawhin and Zanolin in \cite{CaMaZa-92} using Mawhin's coincidence degree and an approximation argument
based on Kupka-Smale theorem. A generalization of Theorem~\ref{th-CMZ} to the neutral functional differential equation
$\tfrac{d}{dt} (Dx_{t}) = g(x_{t})$ was obtained by Bartsch and Mawhin in \cite{BaMa-91} using the topological degree for
$S^{1}$-equivariant maps. The main result in the article of Bartsch and Mawhin (cf.~\cite[Theorem~1]{BaMa-91})
concerns a vector field defined on the whole space, however its proof is based on a ``local'' result
\cite[Theorem~2]{BaMa-91}, which, if adapted to our situation, allows to prove easily the following version of Theorem~\ref{th-CMZ}.

\begin{theorem}\label{th-BM}
Let $\Omega \subseteq X$ be an open bounded set with $\overline{\Omega} \subseteq \text{\rm dom}\,M_{0}$.
Assume that there is no $u\in\partial\Omega$ such that $u'=f_{0}(u)$.
Then
\begin{equation*}
D_{L}(L-M_{0},\Omega) = (-1)^{d} \, \text{\rm deg}_{B}(f_{0},\Omega\cap\mathbb{R}^{d},0).
\end{equation*}
\end{theorem}

\medskip

The next theorem is a generalization of the above results, dealing with an open (possibly unbounded) set $\Omega \subseteq X$.

\begin{theorem}\label{th-CMZ-BM}
Let $\Omega \subseteq \text{\rm dom}\,M_{0}$ be an open (possibly unbounded) set such that the set of the ($T$-periodic) solutions of \eqref{eq-g0} in $\Omega$ is a compact subset of $\Omega$.
Then
\begin{equation*}
D_{L}(L-M_{0},\Omega) = (-1)^{d} \, \text{\rm deg}_{B}(f_{0},\Omega\cap\mathbb{R}^{d},0).
\end{equation*}
\end{theorem}

\begin{proof}
Let $\mathcal{K}$ be the set of the ($T$-periodic) solutions of \eqref{eq-g0}
in $\Omega$. We stress that $\mathcal{K}$ is a compact subset of $\Omega$ (by hypothesis).
Then, we can find an open bounded set $\Omega_{0}$ such that
$\mathcal{K} \subseteq \Omega_{0} \subseteq \text{\rm cl}(\Omega_{0}) \subseteq \Omega$
and so, by the excision,
$D_{L}(L-M_{0},\Omega) = D_{L}(L-M_{0},\Omega_{0})$.
Since there is no $u\in\partial\Omega_{0}$ such that $u'=f_{0}(u)$, we apply Theorem~\ref{th-BM} and get
$D_{L}(L-M_{0},\Omega_{0}) = (-1)^{d} \, \text{\rm deg}_{B}(f_{0},\Omega_{0}\cap\mathbb{R}^{d},0)$.
We conclude observing that $\text{\rm deg}_{B}(f_{0},\Omega_{0}\cap\mathbb{R}^{d},0) = \text{\rm deg}_{B}(f_{0},\Omega\cap\mathbb{R}^{d},0)$
(due to the excision property and the fact that $\mathcal{K} \subseteq \Omega_{0} \subseteq \Omega$).
\end{proof}

\bibliographystyle{elsart-num-sort}
\bibliography{Feltrin_Zanolin_biblio}

\bigskip
\begin{flushleft}

{\small{\it Preprint}}

{\small{\it October 2016}}

\end{flushleft}

\end{document}